\pdfoutput=1
\RequirePackage{ifpdf}
\ifpdf 
\documentclass[pdftex]{sigma}
\else
\documentclass{sigma}
\fi

\usepackage{tikz}
\usetikzlibrary{positioning}

\numberwithin{equation}{section}

\newtheorem{Theorem}{Theorem}[section]
\newtheorem*{Theorem*}{Theorem}

\newtheorem{Lemma}[Theorem]{Lemma}
\newtheorem{Proposition}[Theorem]{Proposition}
\newtheorem{Claim}[Theorem]{Claim} 
 { \theoremstyle{definition}

\newtheorem{Remark}[Theorem]{Remark} }
\newtheorem{Assumption}[Theorem]{Assumption} 

\newcommand{\basichypergeometricseries}[5]{{}_{#1}\phi_{#2}\!\left(\left.\!\!\!\begin{array}{c} #3 \\ #4 \end{array}\!\!\right| #5 \right)}
\newcommand{\gauss}[2]{\genfrac{[}{]}{0pt}{}{#1}{#2}}

\renewcommand{\b}{q}
\newcommand{\q}{t}

\begin{document}
\allowdisplaybreaks

\newcommand{\arXivNumber}{2106.14497}

\renewcommand{\PaperNumber}{104}

\FirstPageHeading

\ShortArticleName{Scaling Limits for the Gibbs States on Distance-Regular Graphs with Classical Parameters}

\ArticleName{Scaling Limits for the Gibbs States \\ on Distance-Regular Graphs\\ with Classical Parameters}

\Author{Masoumeh KOOHESTANI~$^{\rm a}$, Nobuaki OBATA~$^{\rm b}$ and Hajime TANAKA~$^{\rm b}$}

\AuthorNameForHeading{M.~Koohestani, N.~Obata and H.~Tanaka}

\Address{$^{\rm a)}$~Department of Mathematics, K.N.~Toosi University of Technology, Tehran 16765-3381, Iran}
\EmailD{\href{mailto:m.kuhestani@email.kntu.ac.ir}{m.kuhestani@email.kntu.ac.ir}}

\Address{$^{\rm b)}$~Research Center for Pure and Applied Mathematics, Graduate School of Information Sciences,\\
\hphantom{$^{\rm b)}$}~Tohoku University, Sendai 980-8579, Japan}
\EmailD{\href{mailto:obata@tohoku.ac.jp}{obata@tohoku.ac.jp}, \href{mailto:htanaka@tohoku.ac.jp}{htanaka@tohoku.ac.jp}}
\URLaddressD{\url{https://www.math.is.tohoku.ac.jp/~obata/},\newline
\hspace*{14.5mm}\url{https://www.math.is.tohoku.ac.jp/~htanaka/}}

\ArticleDates{Received July 19, 2021, in final form November 22, 2021; Published online November 26, 2021}

\Abstract{We determine the possible scaling limits in the quantum central limit theorem with respect to the Gibbs state, for a growing distance-regular graph that has so-called \emph{classical parameters} with base unequal to one. We also describe explicitly the corresponding weak limits of the normalized spectral distribution of the adjacency matrix. We demonstrate our results with the known infinite families of distance-regular graphs having classical parameters and with unbounded diameter.}


\Classification{46L53; 60F05; 05E30}

\section{Introduction}\label{sec: introduction}

\emph{Quantum probability theory} is a non-commutative extension of classical probability theory; see, e.g., \cite{ALV2002B,Gudder1988B,Meyer1993B,NS2006B,Parthasarathy1992B,VDN1992B}.
This paper is a contribution to the spectral analysis of growing graphs from the viewpoint of this theory.
The adjacency matrix of a graph is regarded as a random variable in this context.
(Formal definitions will be given in Section~\ref{sec: DRGs and CLT}.)
As in many previous works on this topic, our focus will be on central limit theorems (CLTs) for growing graphs.
Of particular interest are growing Cayley graphs.
For example, generators of free groups give rise to \emph{free independent} random variables in the sense of Voiculescu, and we obtain the Wigner semicircle law; cf.~\cite{VDN1992B}.
Another important class of graphs to consider here is that of \emph{distance-regular graphs} \cite{BI1984B,Biggs1993B,BCN1989B,DKT2016EJC}, which generalize distance-transitive (i.e., two-point homogeneous) graphs.
Among many other applications and links, these graphs have been often used as test instances for problems related to random walks on general graphs; see \cite{DKT2016EJC} and the references therein.
Hora \cite{Hora1998IDAQPRT} proved CLTs for several families of distance-regular graphs, including the Hamming graphs (which are also Cayley graphs) and the Johnson graphs, and obtained various distributions, such as the Gaussian, Poisson, geometric, and the exponential distributions.
The CLTs in \cite{Hora1998IDAQPRT} are with respect to the vacuum state, and Hora \cite{Hora2000PTRF} then considered the \emph{Gibbs state}, also known as the \emph{deformed vacuum state}, and extended the CLT for the Johnson graphs and the Hamming graphs.
Later it turned out that distance-regular graphs are particularly well-suited for the method of \emph{quantum decomposition} of the adjacency matrix, which has been playing a key role in obtaining CLTs.
This method was first introduced by Hashimoto \cite{Hashimoto2001IDAQPRT} for certain growing Cayley graphs, and then developed and reformulated further by Hora, Obata, and others; see, e.g.,~\cite{HHO2003JMP,HOT2001P,Hora2004IIS,HO2007B,HO2008TAMS,Obata2017B}.
It not only made the whole theory transparent, but also lead to \emph{quantum} central limit theorems (QCLTs), in which we take into account the three non-commuting components in the quantum decomposition.

Our goal in this paper is to establish QCLTs for those distance-regular graphs said to have \emph{classical parameters}.
For such a graph, the structure of the adjacency algebra is described by just three parameters denoted by $\b$, $\alpha$, and~$\beta$, together with the diameter $d$.
The parameter $\b$, called the \emph{base}, is known to be an integer distinct from $0$ and $-1$, provided that $d\geqslant 3$.
Having classical parameters may look like quite a strong restriction, but it is in some sense rather common among distance-regular graphs.
In fact, except the cycles which we view as trivial, all the \emph{known} infinite families of distance-regular graphs with unbounded diameter either have classical parameters or are closely related to those having classical parameters (by means of halving, folding, doubling, twisting, etc.); cf.~\cite[Section~3]{DKT2016EJC}.
See also \cite{Terwilliger199*M} and \cite[Theorem~6.3]{Tanaka2011EJC} for geometric characterizations of this property.
The classification of distance-regular graphs having classical parameters with $\b=1$ is already complete, and there exist only four infinite families: the Hamming graphs, Doob graphs, halved cubes, and the Johnson graphs.
These graphs were discussed by Hora \cite{Hora1998IDAQPRT,Hora2000PTRF} (see also \cite{HO2007B}) in detail,\footnote{The Doob graphs were not mentioned in \cite{Hora1998IDAQPRT,Hora2000PTRF}, but they have the same classical parameters as certain Hamming graphs, and thus separate discussions are not necessary.} so we will consider the graphs with $\b\ne 1$ in this paper.
Our main result is Theorem \ref{classification} in Section~\ref{sec: DRGs with classical parameters}, where we let $d\rightarrow\infty$ and determine the possible scaling limits for QCLTs in the Gibbs state, in terms of the behaviors of the classical parameters and one other parameter associated with the Gibbs state.
The corresponding weak limits of the normalized spectral distributions will be described explicitly in Section~\ref{sec: description of limit distributions}.
Currently there are fifteen \emph{known} infinite families of distance-regular graphs having classical parameters with unbounded diameter, eleven of which are such that $\b\ne 1$.
Our results apply to these eleven families with $\b\ne 1$, but we stress that they will also be equally applicable whenever we find a new such infinite family in the future.\footnote{The \emph{twisted Grassmann graphs} are the last of these fifteen families and were discovered by Van Dam and Koolen \cite{DK2005IM} in 2005.}
We may also remark that four out of the eleven families are Cayley graphs.

The contents of the other sections are as follows.
In Section~\ref{sec: DRGs and CLT}, we review basic definitions and concepts about algebraic probability spaces and distance-regular graphs, and then recall the QCLT for distance-regular graphs.
Our account follows \cite{HO2007B}.
We will also sharpen the QCLT slightly (Proposition~\ref{remark on QCLT}), by showing that some assumption is redundant.
Section~\ref{sec: more background} is another preliminary section and is devoted to establishing formulas which are needed in Section~\ref{sec: description of limit distributions}.
Section~\ref{sec: examples} discusses concrete examples from the eleven infinite families with $\b\ne 1$ mentioned above.

\section[Algebraic probability spaces, distance-regular graphs, and the quantum central limit theorem]{Algebraic probability spaces, distance-regular graphs,\\ and the quantum central limit theorem}\label{sec: DRGs and CLT}

An \emph{algebraic probability space} is a pair $(\mathcal{A},\varphi)$, where $\mathcal{A}$ is a $*$-algebra over $\mathbb{C}$ and $\varphi\colon \mathcal{A}\rightarrow\mathbb{C}$ is a \emph{state}, i.e., a linear map such that $\varphi(1_{\mathcal{A}})=1$ and that $\varphi(a^*a)\geqslant 0$ for every $a\in\mathcal{A}$, where $1_{\mathcal{A}}$ denotes the identity of $\mathcal{A}$; cf.~\cite[Section~1.1]{HO2007B}.
The elements of $\mathcal{A}$ are called (\emph{algebraic}) \emph{random variables}.
We call $a\in\mathcal{A}$ \emph{real} if $a^*=a$.
For every real random variable $a\in\mathcal{A}$, there exists a~Borel probability measure~$\mu$ on $\mathbb{R}$ (cf.~\cite[Section~1.3]{Bogachev2007B}) such that
\begin{gather}\label{distribution}
	\varphi(a^i)=\int_{-\infty}^{+\infty} \xi^i\mu({\rm d}\xi), \qquad i=0,1,2,\dots.
\end{gather}
We note that such a measure $\mu$ may not be unique.
Several sufficient conditions are known on its uniqueness, such as Carleman's moment test; cf.~\cite[Theorem 1.36]{HO2007B}.

We are interested in algebraic probability spaces arising from graphs.
All the graphs we consider in this paper are finite and simple.
Thus, by a \emph{graph} we mean a pair $\Gamma=(X,R)$ consisting of a non-empty finite set $X$ and a subset $R$ of $\binom{X}{2}$, the set of two-element subsets of~$X$.
The elements of $X$ are \emph{vertices} of $\Gamma$, and the elements of $R$ are \emph{edges} of $\Gamma$.
Two vertices $x,y\in X$ are called \emph{adjacent} (and written $x\sim y$) if $\{x,y\}\in R$.
The \emph{degree} (or \emph{valency}) $k(x)$ of~$x\in X$ is the number of vertices adjacent to $x$.
We call $\Gamma$ $k$-\emph{regular} if $k(x)=k$ for all $x\in X$.
A \emph{path of length $n$ joining} $x,y\in X$ is a sequence of vertices $x=x_0$, $x_1 $, \dots, $x_n=y$ such that $x_{j-1}\sim x_j$ for $j=1,2,\dots,n$.
We will only consider \emph{connected} graphs, i.e., those graphs in which any two vertices are joined by a path.
The \emph{distance} $\partial(x,y)$ of $x,y\in X$ is the length of a shortest path joining them.
The \emph{diameter} of $\Gamma$ is defined by $d=\max\{\partial(x,y)\colon x,y\in X\}$.
Let $M_X(\mathbb{C})$ denote the $\mathbb{C}$-algebra consisting of complex matrices with rows and columns indexed by $X$.
The \emph{adjacency matrix} $A$ of $\Gamma$ is the matrix in $M_X(\mathbb{C})$ defined by
\begin{gather*}
A_{x,y}=\begin{cases} 1 & \text{if} \ x\sim y, \\ 0 & \text{otherwise},\end{cases} \qquad
x,y\in X.
\end{gather*}
By an \emph{eigenvalue} of $\Gamma$, we mean an eigenvalue of $A$.
Likewise, we speak of the \emph{spectrum} of $\Gamma$.

As usual, we view $M_X(\mathbb{C})$ as a $*$-algebra by letting $*$ mean adjoint (i.e., conjugate-transpose).
Associated with the graph $\Gamma$ above is the \emph{adjacency algebra} $\mathcal{A}(\Gamma)$, i.e., the commutative $*$-sub\-algebra of $M_X(\mathbb{C})$ generated by $A$.
Below we give three examples of states for $\mathcal{A}(\Gamma)$.

\medskip\noindent
\textbf{The tracial state.}
This is defined by
\begin{gather*}
	\varphi_{\operatorname{tr}}(B)=\frac{1}{|X|}\operatorname{tr}(B), \qquad B\in\mathcal{A}(\Gamma).
\end{gather*}
For this state, the Borel probability measure $\mu$ from \eqref{distribution} for the random variable $A$ is unique and is the \emph{spectral distribution} of $\Gamma$ given by
\begin{gather*}
\mu(\theta_i)=\frac{m_i}{|X|}, \qquad i=0,1,\dots,e,
\end{gather*}
where $\theta_0,\theta_1,\dots,\theta_e$ are the distinct eigenvalues of $\Gamma$, and $m_i$ denotes the multiplicity of $\theta_i$ in the spectrum of $\Gamma$.

\medskip\noindent
\textbf{The vacuum state.}
Fix a ``base vertex'' $o\in X$, and let
\begin{gather*}
\varphi_0(B)=\langle \hat{o}, B\hat{o}\rangle=B_{o,o}, \qquad B\in\mathcal{A}(\Gamma),
\end{gather*}
where $\hat{o}$ denotes the column vector indexed by $X$ with a $1$ in the $o$ coordinate and $0$ in all other coordinates, and $\langle \cdot,\cdot\rangle$ denotes the standard Hermitian inner product.

\medskip\noindent
\textbf{The Gibbs state.}
This generalizes $\varphi_0$ above.
Let $\q\in\mathbb{R}$, and let
\begin{gather*}
	\varphi_{\q}(B)=\sum_{x\in X} \q^{\partial(x,o)} \langle \hat{x},B\hat{o}\rangle=\sum_{x\in X} \q^{\partial(x,o)} B_{x,o}, \qquad B\in\mathcal{A}(\Gamma),
\end{gather*}
where $0^0:=1$.
The Gibbs state is also called the \emph{deformed vacuum state}.
The scalar $-\log t$ (when $t\geqslant 0$) corresponds to the inverse temperature parameter in the case of the Gibbs state on a canonical ensemble.
It should be remarked however that, unlike the first two examples, the Gibbs state is not always a state.\footnote{For this reason, it would be more appropriate to call $\varphi_{\q}$, say, the Gibbs \emph{functional}.}
See Lemma \ref{when Gibbs state is state} below.

Recall that $\Gamma=(X,R)$ is assumed to be connected with diameter $d$.
For every $i=0,1,\dots,d$, let $A_i$ be the $i^{\mathrm{th}}$ \emph{distance matrix} of $\Gamma$, i.e.,
\begin{gather*}
	(A_i)_{x,y}=\begin{cases} 1 & \text{if} \ \partial(x,y)=i, \\ 0 & \text{otherwise},\end{cases} \qquad x,y\in X.
\end{gather*}
In particular, $A_0=I$ (the identity matrix) and $A_1=A$.
We call $\Gamma$ \emph{distance-regular} if there exist non-negative integers $a_i$, $b_i$, $c_i$, $i=0,1,\dots,d$, such that $b_d=c_0=0$, and that
\begin{gather}\label{3-term recurrence}
	AA_i=b_{i-1}A_{i-1}+a_iA_i+c_{i+1}A_{i+1}, \qquad i=0,1,\dots,d,
\end{gather}
where $b_{-1}A_{-1}=c_{d+1}A_{d+1}:=0$.
We note in this case that $\Gamma$ is $k$-regular with $k=b_0$, $a_0=0$, $c_1=1$,
\begin{gather}\label{a+b+c=k}
	a_i+b_i+c_i=k, \qquad i=0,1,\dots,d,
\end{gather}
and that $b_{i-1}c_i\ne 0$, $i=1,2,\dots,d$.
Moreover, the matrix $A_i$ has constant row and column sum~$k_i$ (which is the number of vertices at distance $i$ from any given vertex), where
\begin{gather}\label{k_i}
	k_i = \frac{b_0b_1\cdots b_{i-1}}{c_1c_2\cdots c_i}, \qquad i=0,1,\dots,d.
\end{gather}
Note that $k_0=1$ and $k_1=k$.

From now on, suppose that $\Gamma$ is distance-regular.
It follows from \eqref{3-term recurrence} that
\begin{gather*}
	\mathcal{A}(\Gamma)=\operatorname{span}\{A_0,A_1,\dots,A_d\},
\end{gather*}
from which it follows that $\dim\mathcal{A}(\Gamma)=d+1$, and hence that $\Gamma$ has exactly $d+1$ distinct eigenvalues $k=\theta_0,\theta_1,\dots,\theta_d$.
Moreover, every matrix in $\mathcal{A}(\Gamma)$ has constant diagonal entries, and hence $\varphi_{\operatorname{tr}}=\varphi_0$.
The Gibbs state $\varphi_{\q}$ is also independent of the base vertex $o\in X$, and we have
\begin{gather}\label{Gibbs state as trace}
	\varphi_{\q}(B)=\frac{1}{|X|}\operatorname{tr}(K_{\q}B), \qquad B\in\mathcal{A}(\Gamma),
\end{gather}
where
\begin{gather*}
	K_{\q}=\bigl(\q^{\partial(x,y)}\bigr)_{x,y\in X}=A_0+\q A_1+\q^2A_2+\dots+\q^dA_d\in\mathcal{A}(\Gamma).
\end{gather*}
From this observation we immediately see that

\begin{Lemma}\label{when Gibbs state is state}
If $\Gamma$ is distance-regular, then the Gibbs state $\varphi_{\q}$ is a state on $\mathcal{A}(\Gamma)$ if and only if the matrix $K_{\q}$ is positive semidefinite.
\end{Lemma}

\noindent
It also follows from \eqref{3-term recurrence}--\eqref{Gibbs state as trace} that the mean and the variance of the random variable $A$ in the Gibbs state $\varphi_{\q}$ are given respectively by (cf.~\cite[Lemma 3.25]{HO2007B})
\begin{gather}\label{mean, variance}
	\varphi_{\q}(A)=\q k, \qquad \Sigma_{\q}^2(A)=\varphi_{\q}\big((A-\q kI)^2\big)=k(1-\q)(1+\q+\q a_1).
\end{gather}

In view of Lemma~\ref{when Gibbs state is state}, we consider the following subset of $\mathbb{R}$:
\begin{gather}\label{pi}
	\pi(\Gamma)=\{\q\in\mathbb{R}\colon K_{\q} \ \text{is positive semidefinite} \}.
\end{gather}
We always have $0,1\in\pi(\Gamma)$, so that $\pi(\Gamma)\ne\varnothing$.
Moreover, by looking at the $2\times 2$ principal submatrices of $K_{\q}$, it follows that
\begin{gather*}
	\pi(\Gamma)\subset [-1,1].
\end{gather*}

With respect to the base vertex $o\in X$, we have the following \emph{quantum decomposition} of $A$:
\begin{gather*}
	A=A^++A^-+A^{\circ},
\end{gather*}
where
\begin{gather*}
	(A^{\epsilon})_{x,y}=\begin{cases} 1 & \text{if} \ x\sim y, \ \partial(x,o)=\partial(y,o)+i_{\epsilon}, \\ 0 & \text{otherwise}, \end{cases} \qquad x,y\in X
\end{gather*}
for $\epsilon\in\{+,-,\circ\}$, with $i_+=1$, $i_-=-1$, and $i_{\circ}=0$.
The matrices $A^+$, $A^-$, and $A^{\circ}$ are called the \emph{quantum components} of $A$ with respect to $o$.
Consider the $\mathbb{C}$-vector space
\begin{gather}\label{primary module}
	W(\Gamma)=\operatorname{span}\{\Phi_0,\Phi_1,\dots,\Phi_d\},
\end{gather}
where the $\Phi_i$ are the unit column vectors given by
\begin{gather*}
	\Phi_i=\frac{1}{\sqrt{k_i}}A_i\hat{o}, \qquad i=0,1,\dots,d.
\end{gather*}
Note that $\Phi_0=\hat{o}$.
With this notation, we have
\begin{gather}\label{Gibbs state in primary module}
	\varphi_{\q}(B)=\sum_{i=0}^d \q^i \sqrt{k_i} \langle \Phi_i,B\Phi_0\rangle, \qquad B\in\mathcal{A}(\Gamma).
\end{gather}
It follows from \eqref{3-term recurrence} and \eqref{k_i} that
\begin{gather}\label{actions of quantum components}
	A^+\Phi_i=\sqrt{c_{i+1}b_i}\Phi_{i+1}, \qquad A^-\Phi_i=\sqrt{c_ib_{i-1}}\Phi_{i-1}, \qquad A^{\circ}\Phi_i=a_i\Phi_i
\end{gather}
for $i=0,1,\dots,d$, where $\sqrt{c_{d+1}b_d}\Phi_{d+1}=\sqrt{c_0b_{-1}}\Phi_{-1}:=0$.
In particular, we observe that the actions of these quantum components on $W(\Gamma)$ are independent of the base vertex $o\in X$.

\begin{Remark}\label{Terwilliger algebra}
The subalgebra $\tilde{\mathcal{A}}(\Gamma)$ of $M_X(\mathbb{C})$ generated by the quantum components $A^+$, $A^-$, and $A^{\circ}$ of $A$ is non-commutative unless $|X|=1$, and is contained in the \emph{Terwilliger algebra} of~$\Gamma$ with respect to $o$; cf.~\cite{Terwilliger1992JAC,Terwilliger1993JACa,Terwilliger1993JACb}.
See \cite{TZ2019JCTA} for discussions on when the two algebras are equal.
The space $W(\Gamma)$ is an irreducible module of the Terwilliger algebra, called the \emph{primary module}.
\end{Remark}

We now recall the QCLT for a growing distance-regular graph in the Gibbs state $\varphi_{\q}$ established in \cite[Section~3.4]{HO2007B}.
For the rest of this paper, let $\Lambda$ be an infinite directed set, and let $(\Gamma_{\lambda})_{\lambda\in\Lambda}$ be a net of distance-regular graphs; see, e.g., \cite[Chapter 2]{Kelley1975B}.
To simplify the notation, we will mostly omit the subscript ``$\lambda$''.
We will view $X$, $d$, $k$, $a_i$, $b_i$, $c_i$, etc., as functions of $\Gamma$.
The scalar $\q\in\pi(\Gamma)$ is chosen and fixed for each of the $\Gamma$ so that the variance $\Sigma_{\q}^2(A)>0$ (cf.~\eqref{mean, variance}), and we will also think of $\q$ as a function of $\Gamma$.
In this paper we are mainly interested in limit distributions with infinite supports, and hence we will assume that
\begin{gather}\label{d tends to infinity}
	d\rightarrow\infty.
\end{gather}
(That is, $\lim\limits_{\lambda\in\Lambda}d(\Gamma_{\lambda})=\infty$.)

In view of \eqref{mean, variance}, we work with the following normalization when taking the limit:
\begin{gather}\label{normalized variable}
	\frac{A-\q kI}{\Sigma_{\q}(A)}=\tilde{A}^++\tilde{A}^-+\tilde{A}^{\circ},
\end{gather}
where
\begin{gather*}
	\tilde{A}^{\pm}=\frac{A^{\pm}}{\Sigma_{\q}(A)}, \qquad \tilde{A}^{\circ}=\frac{A^{\circ}-\q kI}{\Sigma_{\q}(A)}.
\end{gather*}
From \eqref{actions of quantum components} it follows that
\begin{gather*}
	\tilde{A}^+\Phi_i=\sqrt{\overline{\omega}_{i+1}}\Phi_{i+1}, \qquad \tilde{A}^-\Phi_i=\sqrt{\overline{\omega}_i}\Phi_{i-1}, \qquad \tilde{A}^{\circ}\Phi_i=\overline{\alpha}_{i+1}\Phi_i
\end{gather*}
for $i=0,1,\dots,d$, where $\sqrt{\overline{\omega}_{d+1}}\Phi_{d+1}=\sqrt{\overline{\omega}_0}\Phi_{-1}:=0$, and
\begin{gather*}
	\overline{\omega}_i=\frac{c_ib_{i-1}}{\Sigma_{\q}^2(A)}, \qquad i=1,2,\dots,d, \qquad \overline{\alpha}_i=\frac{a_{i-1}-\q k}{\Sigma_{\q}(A)}, \qquad i=1,2,\dots,d+1.
\end{gather*}
We also define the scalar $\overline{\gamma}_i$ by (cf.~\eqref{Gibbs state in primary module})
\begin{gather*}
	\overline{\gamma}_i=\q^i\sqrt{k_i}, \qquad i=0,1,\dots,d.
\end{gather*}

Consider the following limits:
\begin{gather*}
	\overline{\omega}_i\rightarrow \omega_i, \qquad \overline{\alpha}_i\rightarrow\alpha_i, \qquad i=1,2,\dots, \qquad \overline{\gamma}_i \rightarrow\gamma_i, \qquad i=0,1,\dots.
\end{gather*}
These limits do not necessarily exist in general, and we impose the following:

\begin{Assumption}\label{condition (DR)}
With the above situation, we assume that the limits $\omega_i$, $\alpha_i$, and $\gamma_i$ exist and that $\omega_i>0$ for all $i$.
We note that $\gamma_0=1$.
\end{Assumption}

With reference to Assumption \ref{condition (DR)}, let $\mathcal{W}$ be an infinite-dimensional $\mathbb{C}$-vector space with a~fixed basis $\{\Psi_i\colon i=0,1,\dots\}$, where we equip $\mathcal{W}$ with the Hermitian inner product $\langle \cdot,\cdot\rangle$ for which the $\Psi_i$ are orthonormal.
We define the linear operators $B^+$, $B^-$, and $B^{\circ}$ on $\mathcal{W}$ by
\begin{gather*}
	B^+\Psi_i=\sqrt{\omega_{i+1}}\Psi_{i+1}, \qquad B^-\Psi_i=\sqrt{\omega_i}\Psi_{i-1}, \qquad B^{\circ}\Psi_i=\alpha_{i+1}\Psi_i
\end{gather*}
for $i=0,1,\dots$, where $\sqrt{\omega_{-1}}\Psi_{-1}:=0$.
Note that $B^+$ and $B^-$ are adjoints of each other.
The quadruple $(\mathcal{W},\{\Psi_i\},B^+,B^-)$ is called the \emph{interacting Fock space} (of one mode) \emph{associated with Jacobi sequence} $\{\omega_i\}$.

Recall the non-commutative algebra $\tilde{\mathcal{A}}(\Gamma)$ from Remark \ref{Terwilliger algebra}, and observe that $\tilde{A}^+$, $\tilde{A}^-$, and~$\tilde{A}^{\circ}$ generate $\tilde{\mathcal{A}}(\Gamma)$.
We now extend the domain of the Gibbs state $\varphi_{\q}$ to $\tilde{\mathcal{A}}(\Gamma)$; cf.~\eqref{Gibbs state in primary module}.
This extension is again independent of the base vertex $o\in X$, but it should be remarked that it may fail to be a state on $\tilde{\mathcal{A}}(\Gamma)$ (though it is indeed a state on $\mathcal{A}(\Gamma)$ by Lemma \ref{when Gibbs state is state}).
The QCLT in the Gibbs state is stated as follows:

\begin{Theorem}[{\cite[Theorem 3.29]{HO2007B}}]\label{QCLT}
With reference to Assumption $\ref{condition (DR)}$, we have
\begin{gather*}
	\varphi_{\q}\big(\tilde{A}^{\epsilon_m}\cdots\tilde{A}^{\epsilon_1}\big) \rightarrow \sum_{i=0}^{\infty} \gamma_i \langle \Psi_i,B^{\epsilon_m}\cdots B^{\epsilon_1}\Psi_0\rangle
\end{gather*}
for any $\epsilon_1,\dots,\epsilon_m\in\{+,-,\circ\}$ and $m=1,2,\dots$.
\end{Theorem}

\begin{Remark}\label{limit distribution}
There exists a Borel probability measure $\mu_{\infty}$ on $\mathbb{R}$ such that (cf.~\eqref{distribution})
\begin{gather*}
	\sum_{i=0}^{\infty} \gamma_i \langle \Psi_i,(B^++B^-+B^{\circ})^m\Psi_0\rangle=\int_{-\infty}^{+\infty} \!\! \xi^m\mu_{\infty}({\rm d}\xi), \qquad m=1,2,\dots.
\end{gather*}	
This $\mu_{\infty}$ is called the \emph{asymptotic normalized spectral distribution} of $A$ in the Gibbs state, and we are interested in finding and describing it.
See Section~\ref{sec: description of limit distributions}.
\end{Remark}

We end this section with some comments.
In \cite[Section~3.4]{HO2007B}, Hora and Obata also considered the case when $\omega_m=0$ for some $m$, so that the probability measure $\mu_{\infty}$ above has finite support.
However, if we stick to the case when $\omega_i>0$ for all $i$ as in Assumption \ref{condition (DR)}, then assuming the existence of the $\gamma_i$ turns out to be redundant.
To be more precise, we show the following:

\begin{Proposition}\label{remark on QCLT}
Suppose that the $\omega_i$ and the $\alpha_i$ exist and that $\omega_i>0$ for all $i$.
Then the $\gamma_i$ exist as well.
In particular, Theorem $\ref{QCLT}$ holds true under this weaker assumption.
\end{Proposition}

Proposition \ref{remark on QCLT} is a consequence of Claims \ref{two possible accumulation points}--\ref{exactly one accumulation point} below.
For the rest of this section, we assume the existence of the $\omega_i>0$ and that of the $\alpha_i$.

First, observe that
\begin{gather}\label{omega_1 exists}
	(1-\q)(1+\q+\q a_1)=\frac{\Sigma_{\q}^2(A)}{k}=\frac{1}{\overline{\omega}_1} \rightarrow \frac{1}{\omega_1}>0.
\end{gather}
Since $a_0=0$, the existence of $\gamma_1$ follows from that of $\alpha_1$ and \eqref{omega_1 exists}: $\gamma_1=-\alpha_1/\sqrt{\omega_1}$.
Note that if $k=2$ then $\Gamma$ is either the $2d$-cycle or the $(2d+1)$-cycle.
Bang, Dubickas, Koolen, and Moulton~\cite{BDKM2015AM} proved the \emph{Bannai--Ito conjecture}:

\begin{Theorem}[{\cite{BDKM2015AM}}]\label{BI conjecture}
There exist only finitely many distance-regular graphs for each fixed degree $k\geqslant 3$.
\end{Theorem}

Since we are letting $d\rightarrow\infty$ (cf.~\eqref{d tends to infinity}), it follows that

\begin{Claim}\label{two possible accumulation points}
If $\xi$ is an accumulation point of $1/k\in(0,1/2]$, then $\xi\in\{0,1/2\}$.
\end{Claim}

We next handle each of the two possible accumulation points of $1/k$.

\begin{Claim}\label{case 1/2}
Suppose that $1/2$ is an accumulation point of $1/k$, and consider a subnet of $(\Gamma_{\lambda})_{\lambda\in \Lambda}$ for which $k=2$ eventually.
Then the $\gamma_i$ exist on this subnet.
Moreover, we have $\omega_i=\omega_1/2$ and~$\alpha_i=\alpha_1$ for $i=2,3,\dots$.
\end{Claim}

\begin{proof}
Recall that $\gamma_1$ exists.
Since $k=2$ eventually, this means that $\q$ is convergent on this subnet.
For the cycles, we have $k_i=2$, $i=1,2,\dots,d-1$.
Since $d\rightarrow\infty$, it follows that the $\gamma_i$ all exist on this subnet.
The last statement also follows from \eqref{omega_1 exists} and since the cycles satisfy $(a_i,b_i,c_i)=(0,1,1)$ for $i=1,2,\dots,d-1$.
\end{proof}

\begin{Claim}\label{case 0}
Suppose that $0$ is an accumulation point of $1/k$, and consider a subnet of $(\Gamma_{\lambda})_{\lambda\in \Lambda}$ for which $k\rightarrow\infty$.
Then the $\gamma_i$ exist on this subnet.
Moreover, we have $a_i/\sqrt{k}\rightarrow \alpha_{i+1}/\sqrt{\omega_1}+\gamma_1$, $b_i/k\rightarrow 1$, and $c_i\rightarrow\omega_i/\omega_1$ for $i=1,2,\dots$ on this subnet.
\end{Claim}

\begin{proof}
That $a_i/\sqrt{k}\rightarrow \alpha_{i+1}/\sqrt{\omega_1}+\gamma_1$, $i=1,2,\dots$, is immediate from~\eqref{omega_1 exists}.
We next show that $b_i/k\rightarrow 1$ and $c_i\rightarrow\omega_i/\omega_1$ for $i=1,2,\dots$ on this subnet.
Suppose by induction that $c_i\rightarrow\omega_i/\omega_1$ for some $i$.
We have $a_i=o(k)$ and $c_i=o(k)$ since $k\rightarrow\infty$, and hence $b_i/k\rightarrow 1$ by~\eqref{a+b+c=k}.
Then it follows in turn that $c_{i+1}\sim \overline{\omega_{i+1}}/\overline{\omega_1}\rightarrow\omega_{i+1}/\omega_1$.
The existence of the $\gamma_i$ on this subnet now follows from these comments, \eqref{k_i}, and the existence of $\gamma_1$.
\end{proof}

Finally, we show that the above two cases do not coexist.

\begin{Claim}\label{exactly one accumulation point}
There exists exactly one accumulation point of $1/k\in (0,1/2]$.
More precisely, we have either $k=2$ eventually, or $k\rightarrow\infty$.
\end{Claim}

\begin{proof}
In view of Claim~\ref{two possible accumulation points}, suppose on the contrary that both $0$ and $1/2$ are accumulation points of $1/k$.
On the one hand, we have $\omega_i=\omega_1/2$, $i=2,3,\dots$, by Claim~\ref{case 1/2}.
On the other hand, we have $\omega_i\geqslant \omega_1$, $i=2,3,\dots$, by Claim~\ref{case 0} and since $c_i\geqslant 1$.
This is a contradiction, and the result follows.
\end{proof}

\begin{proof}[Proof of Proposition \ref{remark on QCLT}]
Immediate from Claims \ref{two possible accumulation points}--\ref{exactly one accumulation point}.
\end{proof}

The following is another important consequence of the above discussions:

\begin{Claim}\label{ci is eventually constant}
Each of the $c_i$ is eventually constant.
\end{Claim}

\begin{proof}
Follows from Claims \ref{two possible accumulation points}--\ref{exactly one accumulation point} and since the $c_i$ are integers.
\end{proof}

\begin{Remark}
In \cite[Chapter 7]{HO2007B}, Hora and Obata extended the method of the quantum decomposition and the QCLT to more general growing regular graphs, and obtained some sufficient conditions for the theorem to hold.
See also \cite{HO2008TAMS}.
In particular, these conditions can be applied to Cayley graphs on Coxeter groups, such as the symmetric groups.
For distance-regular graphs, these conditions turn out to be reduced to the following (besides that $\Sigma_{\q}^2(A)>0$): (i) $k\rightarrow\infty$; (ii) each of the $c_i$ is eventually constant; (iii) the $a_i/\sqrt{k}$ are convergent; (iv) $\gamma_1$ exists.
See \cite[Theorem 7.14 and~Proposition 7.17]{HO2007B}.
Therefore, if we focus only on distance-regular graphs with $k\geqslant 3$, then it follows from Claims \ref{two possible accumulation points}--\ref{ci is eventually constant} that these sufficient conditions are in fact equivalent to Assumption \ref{condition (DR)} (or the existence of the $\omega_i>0$ and that of the $\alpha_i$, by virtue of Proposition~\ref{remark on QCLT}).
\end{Remark}

\section{Distance-regular graphs with classical parameters}\label{sec: DRGs with classical parameters}

A distance-regular graph $\Gamma$ with diameter $d$ is said to have \emph{classical parameters} $(d,\b,\alpha,\beta)$ (cf.~\cite[Section~6.1]{BCN1989B}) whenever the $b_i$ and the $c_i$ are expressed as
\begin{gather}\label{b and c}
	b_i=\left(\gauss{d}{1}-\gauss{i}{1}\right)\!\!\left(\beta-\alpha\gauss{i}{1}\right), \qquad c_i=\gauss{i}{1}\!\left(1+\alpha\gauss{i-1}{1}\right)
\end{gather}
for $i=0,1,\dots,d$, where
\begin{gather*}
	\gauss{i}{1}=1+\b+\b^2+\dots+\b^{i-1}
\end{gather*}
is a Gaussian binomial coefficient.
We call $\b$ the \emph{base}.
In particular,
\begin{gather}\label{k}
	k=b_0=\gauss{d}{1}\beta,
\end{gather}
and by \eqref{a+b+c=k} we have
\begin{gather}\label{a}
	a_i=\gauss{i}{1}\!\left(\beta-1+\alpha\left(\gauss{d}{1}-\gauss{i}{1}-\gauss{i-1}{1}\right)\!\right), \qquad i=0,1,\dots,d.
\end{gather}
It is known (see~\cite[Proposition 6.2.1]{BCN1989B}) that
\begin{gather*}
	\b\in\mathbb{Z}\setminus\{0,-1\} \qquad \text{if} \quad d\geqslant 3.
\end{gather*}
As mentioned in Section~\ref{sec: introduction}, all the graphs with $\b=1$ are known, and the QCLTs for them have been obtained, so our aim in this paper is to discuss the case where $\b\in\{\pm 2,\pm 3,\dots\}$.

Suppose that $\Gamma$ has classical parameters $(d,\b,\alpha,\beta)$.
By \cite[Corollary 8.4.2]{BCN1989B}, the $d+1$ distinct eigenvalues of $\Gamma$ are given by
\begin{gather}\label{eigenvalues}
	\theta_i=\frac{b_i}{\b^i}-\gauss{i}{1}=\gauss{d-i}{1}\!\left(\beta-\alpha\gauss{i}{1}\right)-\gauss{i}{1}, \qquad i=0,1,\dots,d.
\end{gather}
For $i=0,1,\dots,d$, let $E_i$ denote the orthogonal projection onto the eigenspace of $A$ for $\theta_i$.
The~$E_i$ are polynomials in $A$, and we have
\begin{gather}\label{second basis}
	\mathcal{A}(\Gamma)=\operatorname{span}\{E_0,E_1,\dots,E_d\}.
\end{gather}
Note by \eqref{k} that $\theta_0=k$.
Since $\Gamma$ is regular and connected, it follows that (cf.~\cite[p.~45]{BCN1989B})
\begin{gather}\label{E0}
	E_0=\frac{1}{|X|}J,
\end{gather}
where $J$ denotes the all-ones matrix.

Recall the set $\pi(\Gamma)$ from \eqref{pi}.
It seems to be a difficult problem to determine $\pi(\Gamma)$ in general.
For the Hamming graphs and the Johnson graphs, which have classical parameters with $\b=1$, it is shown (see~\cite[Propositions 5.16 and 6.27]{HO2007B}) that this set contains the interval $[0,1]$, as consequences of Bo\.{z}ejko's quadratic embedding test; cf.~\cite[Proposition~2.14]{HO2007B}.
For the case $\b\ne 1$, the following result again finds infinitely many elements of $\pi(\Gamma)$:

\begin{Proposition}\label{negative powers of q}
Suppose that $\Gamma$ has classical parameters $(d,\b,\alpha,\beta)$ with $d\geqslant 3$ and $\b\in\{\pm 2,\pm 3,\dots\}$.
Then $\b^{-i}\in\pi(\Gamma)$ for $i=0,1,2,\dots$.
\end{Proposition}

\begin{proof}
We already mentioned that $\b^0=1\in\pi(\Gamma)$.
By \cite[Corollary 8.4.2]{BCN1989B}, the first projection matrix $E_1$ is of the form
\begin{gather*}
	E_1=\frac{1}{|X|}\sum_{i=0}^d \big(\zeta+\eta \b^{-i}\big)A_i=\zeta E_0+\frac{\eta}{|X|}K_{\b^{-1}}
\end{gather*}
for some $\zeta,\eta\in\mathbb{R}$, where we have used \eqref{E0}.
It is customary to write\footnote{\label{*-notation}The $*$-notation here is used to mean ``dual'' objects, and is standard in the theory of distance-regular graphs. The $\theta_i^*$ are referred to as the \emph{dual eigenvalues} of $\Gamma$.} $\theta_i^*=\zeta+\eta \b^{-i}$, $i=0,1,\dots,d$.
Note that $\theta_0^*=\operatorname{tr}(E_1)=m_1$, the multiplicity of $\theta_1$ in the spectrum of $\Gamma$.
It is known (see~\cite[Lemma 2.2.1]{BCN1989B}) that $|\theta_i^*|\leqslant m_1$ for $i=0,1,\dots, d$.
We have $\eta\ne 0$, for otherwise $E_1$ would be a scalar multiple of $E_0$, a contradiction.
If $\eta<0$ then $\eta \b^{-1}>\eta$, so that $\theta_1^*>m_1$, again a~contradiction.
Hence $\eta>0$.

We next show that $\zeta\leqslant 0$.
If $\zeta>0$ and $\b\geqslant 2$ then $E_1$ would be a non-zero non-negative matrix and thus $\operatorname{tr}(E_0E_1)>0$ by \eqref{E0}, which is absurd.
Hence suppose that $\b\leqslant -2$.
We~observe that $\zeta\leqslant 0$ if and only if $\theta_1^*\leqslant m_1/\b$.
By \cite[Lemma 2.2.1]{BCN1989B}, we have $\theta_1/k=\theta_1^*/m_1$.
Using this, \eqref{k}, \eqref{a}, and \eqref{eigenvalues}, we easily verify that $\zeta\leqslant 0$ if and only if $\theta_1\leqslant k/\b$ if and only if $a_1+\b+1\leqslant 0$.
A \emph{kite of length two} in $\Gamma$ is a quadruple $(x,y,z,w)$ of vertices such that $x\sim y\sim w$, $x\sim z\sim w$, $y\sim z$, and $x\not\sim w$:
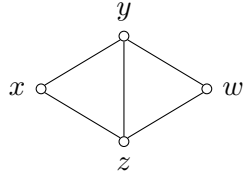
\begin{center}
\begin{tikzpicture}
[bit/.style={circle,inner sep=0mm,minimum size=1.3mm,draw=black}]
\node [bit] (1) [label=left:$x$] {};
\node [bit] (2) [above right=of 1,yshift=-4mm, label=above:$y$] {};
\node [bit] (3) [below right=of 1,yshift=4mm, label=below:$z$] {};
\node [bit] (4) [below right=of 2,yshift=4mm, label=right:$w$] {};
\draw (1.north east) -- (2.south west);
\draw (1.south east) -- (3.north west);
\draw (2.south) -- (3.north);
\draw (2.south east) -- (4.north west);
\draw (3.north east) -- (4.south west);
\end{tikzpicture}
\end{center}
A kite of length two is also called a \emph{parallelogram of length two}.
By \cite[Theorem 2.12]{Terwilliger1995EJC}, $\Gamma$ has no kite of length two.
By \cite[Lemma 3.6]{Weng1997EJC}, we then have $a_1+\b+1\leqslant 0$.
It follows that $\zeta\leqslant 0$.

Since
\begin{gather*}
	K_{\b^{-1}}=\frac{|X|}{\eta}(E_1-\zeta E_0),
\end{gather*}
it follows that $K_{\b^{-1}}$ is positive semidefinite, i.e., $\b^{-1}\in\pi(\Gamma)$.
For $i=2,3,\dots$, we observe that the matrix $K_{\b^{-i}}$ is a principal submatrix of $(K_{\b^{-1}})^{\otimes i}$ since $(K_{\b^{-i}})_{x,y}=((K_{\b^{-1}})_{x,y})^i$ for all $x,y\in X$, and it is therefore positive semidefinite as well.
This completes the proof.
\end{proof}

We comment on the uniqueness of the classical parameters for $\Gamma$.
By \cite[Corollary 6.2.2]{BCN1989B}, the classical parameters $(d,\b,\alpha,\beta)$ for $\Gamma$ are uniquely determined provided that $d\geqslant 3$, with the exception of the pairs
\begin{gather}\label{two expressions}
	\big(d,\ell^2,0,\ell\big), \qquad \bigg(d,-\ell,\frac{\ell(\ell+1)}{1-\ell},\frac{\ell(1+(-\ell)^d)}{1-\ell}\bigg),
\end{gather}
where $\ell\geqslant 2$.
Ivanov and Shpectorov \cite{IS1989LAA} showed that if $\Gamma$ has the above classical parameters then $\ell$ is a prime power and $\Gamma$ is the Hermitian dual polar graph $^2\!A_{2d-1}(\ell)$ (cf.~\cite[Section~9.4]{BCN1989B}).

\begin{Assumption}\label{assume classical parameters}
Recall Assumption $\ref{condition (DR)}$.
We moreover assume that the graph $\Gamma=\Gamma_{\lambda}$ has classical parameters $(d,\b,\alpha,\beta)$ with $d\geqslant 3$ and $\b\in\{\pm 2,\pm 3,\dots\}$.
We will view $\b$, $\alpha$, and~$\beta$ as functions of~$\Gamma$.
For the classical parameters in \eqref{two expressions}, we understand that we may choose either set of them.
\end{Assumption}

Recall that we are assuming that $d\rightarrow\infty$; cf.~\eqref{d tends to infinity}.
Our goal is to describe the limit behaviors of the other parameters $\b$, $\alpha$, and $\beta$.
In the following discussions, we will freely use the expressions \eqref{b and c}, \eqref{k}, and \eqref{a}.
In particular, we note that
\begin{gather}\label{alpha is discrete}
	\alpha=\frac{c_2}{\b+1}-1 \in \frac{1}{\b+1}\mathbb{Z}.
\end{gather}
The cycles ($k=2$) with $d\geqslant 3$ do not have classical parameters, so that it follows from Claim~\ref{exactly one accumulation point} that
\begin{gather}\label{k diverges}
	k\rightarrow\infty.
\end{gather}

\begin{Claim}\label{q is finite}
With reference to Assumption $\ref{assume classical parameters}$, $\limsup |\b|<\infty$.
In particular, $\b$ eventually takes only finitely many values.
\end{Claim}

\begin{proof}
Suppose that $\limsup |\b|=\infty$, or equivalently, $0$ is an accumulation point of $1/\b$.
Then there exists a subnet of $(\Gamma_{\lambda})_{\lambda\in \Lambda}$ for which $|\b|\rightarrow \infty$.
By Claim~\ref{ci is eventually constant}, $c_2$ is eventually constant, so that it follows from \eqref{alpha is discrete} that $\alpha\rightarrow -1$ on this subnet.
This in turn implies that $c_3\rightarrow\infty$ on this subnet, but this is impossible since $c_3$ is also eventually constant by Claim~\ref{ci is eventually constant}.
It follows that $\limsup |\b|<\infty$.
\end{proof}

\begin{Claim}\label{at most 3 accumulation points}
With reference to Assumption $\ref{assume classical parameters}$, suppose that $\b$ is not convergent.
Then the set of accumulation points of $\b$ is of the form $\big\{\ell,\ell^2\big\}$ or $\big\{\ell,-\ell,\ell^2\big\}$ for some $\ell\in\{\pm 2,\pm 3,\dots\}$, where $\alpha\ne 0$ when $\b=\pm \ell$, and $\alpha=0$ when $\b=\ell^2$.
Moreover, we have $a_i/\sqrt{k}\rightarrow 0$ for every $i=1,2,\dots$.
\end{Claim}

\begin{proof}
Recall Claim~\ref{q is finite}.
Since $c_2$ is eventually constant by Claim~\ref{ci is eventually constant}, it follows from \eqref{alpha is discrete} that $\alpha$ is eventually determined by $\b$.
We have
\begin{gather*}
	\frac{c_{i+1}}{c_i}=\frac{\b^{i+1}-1}{\b^i-1}\cdot \frac{\b-1+\alpha(\b^i-1)}{\b-1+\alpha(\b^{i-1}-1)}, \qquad i=1,2,\dots.
\end{gather*}
For sufficiently large $i$ (cf.~\eqref{d tends to infinity}), the RHS can be arbitrarily close to $\b^2$ when $\alpha\ne 0$ and $\b$ when $\alpha=0$.
Let $\ell$ and $\ell'$ be two distinct accumulation points of $\b$, where $|\ell'|\geqslant |\ell|(\geqslant 2)$.
Since the LHS above is eventually constant for every $i$ by Claim~\ref{ci is eventually constant}, it follows that $\ell'\in\big\{{-}\ell,\ell^2\big\}$, where $\alpha\ne 0$ when $\b=\pm \ell$, and $\alpha=0$ when $\b=\ell^2$.

We next show that $a_i/\sqrt{k}\rightarrow 0$ for every $i$.
Recall that $k\rightarrow\infty$ (cf.~\eqref{k diverges}).
The $a_i/\sqrt{k}$ are convergent by Claims \ref{case 0} and \ref{exactly one accumulation point}.
Suppose that $a_i/\sqrt{k}\nrightarrow 0$ for some $i$.
Then we have $a_i\rightarrow\infty$ for this $i$, and since $\b$ and $\alpha$ are eventually bounded, it follows that $\bigl|\beta+\alpha\gauss{d}{1}\bigr|=\Theta(a_i)\rightarrow\infty$, and hence that $a_j\sim \gauss{j}{1}\bigl(\beta+\alpha\gauss{d}{1}\bigr)\sim \gauss{j}{1}a_1$ for all $j$.
On the one hand, this shows that $a_1/\sqrt{k}\nrightarrow 0$.
On the other hand, this also shows that $a_j/\sqrt{k}$ cannot converge whenever $j\geqslant 2$ since $\gauss{j}{1}$ takes at least two values depending on $\b$.
Hence we must have $a_i/\sqrt{k}\rightarrow 0$ for every $i$, as desired.

It remains to show that $\ell^2$ is an accumulation point of $\b$.
There is nothing to prove if $\ell'=\ell^2$, so that we assume that $\ell'=-\ell$.
We have $c_3=\big(\b^2+\b+1\big)(c_2-\b)$ by \eqref{alpha is discrete}.
By setting $\b=\pm\ell$ in this expression and then equating, we find that eventually $c_2=\ell^2+1$.
Choose $\b\in\{\ell,-\ell\}$ with $\b<0$, and recall that $\alpha\ne 0$ in this case.
In particular, we have $a_2\ne \gauss{2}{1}a_1$.
Suppose that $a_1=0$.
Then $a_2\ne 0$, and we have $c_2\leqslant 2$ by \cite[Theorem 2.1]{PW2009JCTB}, but this is impossible since $c_2=\ell^2+1\geqslant 5$.
Hence $a_1\ne 0$.
Then it follows from \cite[Main Theorem]{Weng1999JCTB} that, provided that $d\geqslant 4$, either (i) $\Gamma$ is the dual polar graph $^2\!A_{2d-1}(-\b)$ with $\alpha=\b(\b-1)/(\b+1)$ (cf.~\cite[Section~9.4]{BCN1989B}), or (ii) $\Gamma$ is the Hermitian forms graph $\operatorname{Her}\big(d,\b^2\big)$ with $\alpha=\b-1$ (cf.~\cite[Section~9.5]{BCN1989B}), or (iii) we have $\alpha=(\b-1)/2$ and $\beta=-\big(\b^d+1\big)/2$.
Since $c_2=\b^2+1$, it follows from~\eqref{alpha is discrete} that we are in (i) above.
However, the graph $^2\!A_{2d-1}(-\b)$ has another set of classical parameters $\big(d,\b^2,0,-\b\big)$ (cf.~\eqref{two expressions}), and therefore $\ell^2=\b^2$ must also be an accumulation point.
\end{proof}

\begin{Theorem}\label{classification}
With reference to Assumption $\ref{assume classical parameters}$, $\b$ eventually takes at most three values.
Suppose that $\b$ is eventually constant.
Then so is $\alpha$, and the following hold:
\begin{enumerate}\itemsep=0pt
\item[$(i)$] If $\alpha\ne 0$, then $\beta/\sqrt{k}$ is eventually bounded, and there exist scalars $\gamma$ and $\rho$ with $\rho>0$ and $\gamma(\rho+\alpha/\rho)>-1$, such that $\q\sqrt{k}\rightarrow\gamma$ and the accumulation points of $\beta/\sqrt{k}$ are in $\{\rho,\alpha/\rho\}$.
Moreover, we have $\rho=\sqrt{-\alpha}$ if $\b<0$.
\item[$(ii)$] If $\alpha=0$, then there exist scalars $\gamma$ and $\rho$ with $\rho\geqslant 0$ and $\gamma\rho>-1$, such that $\q\sqrt{k}\rightarrow\gamma$ and~$\beta/\sqrt{k}\rightarrow\rho$.
\end{enumerate}
Suppose that $\b$ is not convergent.
Then there exists a subnet of $(\Gamma_{\lambda})_{\lambda\in \Lambda}$ for which $\b$ is eventually constant and (ii) holds above with $\rho=0$.

Conversely, if $(\Gamma_{\lambda})_{\lambda\in \Lambda}$ is a net of distance-regular graphs having classical parameters with $d\geqslant 3$ and $\b\in\{\pm 2,\pm 3,\dots\}$, where $\b$ and $\alpha$ are eventually constant, such that $d\rightarrow \infty$ and $(i)$ or $(ii)$ holds above with respect to a suitable function $\q\in\pi(\Gamma)$ with $\Sigma_{\q}^2(A)>0$, then $(\Gamma_{\lambda})_{\lambda\in \Lambda}$ satisfies Assumption $\ref{condition (DR)}$ $($and thus Assumption $\ref{assume classical parameters}$ as well$)$.
\end{Theorem}

\begin{proof}
The first statement follows from Claims \ref{q is finite} and \ref{at most 3 accumulation points}.
We also mentioned earlier (cf.~\eqref{k diverges}) that $k\rightarrow\infty$.

Suppose that $\b$ is eventually constant.
That $\alpha$ is eventually constant follows from Claim~\ref{ci is eventually constant} and \eqref{alpha is discrete}.
Set $\gamma=\gamma_1$.
Since it exists, we have $\q\rightarrow 0$.
Recall again that the $a_i/\sqrt{k}$ are convergent by Claims \ref{case 0} and \ref{exactly one accumulation point}, and observe that this is equivalent to saying that $\bigr(\beta+\alpha\gauss{d}{1}\bigr)/\sqrt{k}$ converges, say, to $\sigma$.
Assume that $\alpha\ne 0$, and let $\xi=\rho$ be a root of the equation $\xi+\alpha/\xi=\sigma$ in the variable $\xi$.
Then the other root is $\xi=\alpha/\rho$.
Since $\beta/\sqrt{k}$ and $\gauss{d}{1}/\sqrt{k}$ are inverses of each other, it follows that $\beta/\sqrt{k}$ is eventually bounded, and that $\rho$ and $\alpha/\rho$ are its only possible accumulation points.
If $\b>0$ then $\alpha>0$ and $\beta>0$, so that we must have $\rho>0$.
If $\b<0$ then $\alpha<0$, and since $a_i/\sqrt{k}\rightarrow \gauss{i}{1}\sigma$ for every $i$ and the $\gauss{i}{1}$ alternate in sign, it follows that $\sigma=0$, so that we may take $\rho=\sqrt{-\alpha}>0$ (and thus $\alpha/\rho=-\sqrt{-\alpha}<0$).
By \eqref{omega_1 exists} and since $\q a_1=\q\sqrt{k}\cdot a_1/\sqrt{k}\rightarrow \gamma\sigma=\gamma(\rho+\alpha/\rho)$, we have $\gamma(\rho+\alpha/\rho)>-1$.
Assume next that $\alpha=0$.
We~have $\b>0$ and $\beta>0$ in this case, and set $\rho=\sigma\geqslant 0$.

Suppose that $\b$ is not convergent, and let the integer $\ell$ be as in Claim~\ref{at most 3 accumulation points}.
Then $\ell^2$ is an accumulation point of $\b$, so that there is a subnet of $(\Gamma_{\lambda})_{\lambda\in \Lambda}$ for which eventually $\b=\ell^2$.
Recall by Claim~\ref{at most 3 accumulation points} that eventually $\alpha=0$, and that $a_i/\sqrt{k}\rightarrow 0$ for every $i$.
Hence we are in the second case above with $\rho=\sigma=0$.

Finally, let $(\Gamma_{\lambda})_{\lambda\in \Lambda}$ be a net of distance-regular graphs as described in the last paragraph of the theorem.
Note that $k\geqslant 3$ since the cycles with $d\geqslant 3$ do not have classical parameters.
Hence it follows from Theorem \ref{BI conjecture} that $k\rightarrow\infty$.
Since $\q\sqrt{k}\rightarrow\gamma$, we then have $\q\rightarrow 0$.
If $\alpha\ne 0$ then we must also have $|\beta|\rightarrow\infty$ since $\rho$ and $\alpha/\rho$ are non-zero.
Observe that $a_i/\sqrt{k}\rightarrow\gauss{i}{1}(\rho+\alpha/\rho)$ for every $i$, where we set $0/0:=0$ for brevity.
In particular, we have $\q a_1\rightarrow\gamma(\rho+\alpha/\rho)>-1$, from which it follows that $\omega_1$ exists and is positive.
It is now immediate to verify that the $\omega_i$ all exist and are positive, and that the $\alpha_i$ exist.
From Proposition \ref{remark on QCLT} it also follows that the $\gamma_i$ exist.
\end{proof}

Consider the case when $\b$ is eventually constant in Theorem \ref{classification}, and recall that so is $\alpha$ in this case.
Recall (cf.~\eqref{b and c}) also the formula for the $c_i$.
The scalars $\omega_i$, $\alpha_i$, and $\gamma_i$ are expressed in terms of $\b$, $\alpha$, and the two scalars $\gamma$ and $\rho$ in Theorem \ref{classification}\,$(i)$ and $(ii)$ as
\begin{gather*}
	\omega_i=\frac{c_i}{1+\gamma(\rho+\alpha/\rho)}, \qquad \alpha_i=\frac{\gauss{i-1}{1}(\rho+\alpha/\rho)-\gamma}{\sqrt{1+\gamma(\rho+\alpha/\rho)}}, \qquad i=1,2,\dots,
\end{gather*}
and
\begin{gather*}
	\gamma_i=\frac{\gamma^i}{\sqrt{c_i\cdots c_1}}, \qquad i=0,1,\dots,
\end{gather*}
where we set $0/0:=0$ and $0^0:=1$.

\section{More background on graphs with classical parameters}\label{sec: more background}

In order to describe the asymptotic normalized spectral distributions corresponding to Theorem~\ref{classification}\,$(i)$ and $(ii)$, we collect in this section necessary formulas for distance-regular graphs with classical parameters.
Thus, throughout this section, we let $\Gamma=(X,R)$ denote a (fixed) distance-regular graph with classical parameters $(d,\b,\alpha,\beta)$, where $d\geqslant 3$ and $\b\in\{\pm 2,\pm3,\dots\}$.

Recall the eigenvalues $\theta_i$, $i=0,1,\dots,d$, of $\Gamma$ and the corresponding orthogonal projections~$E_i$, $i=0,1,\dots,d$; cf.~\eqref{eigenvalues} and \eqref{second basis}.
It is known (see~\cite[Corollary~8.4.2]{BCN1989B}) that the ordering $(E_0,E_1,\dots,E_d)$ is $Q$-\emph{polynomial};
that is to say, there exist scalars\footnote{See footnote \ref{*-notation}.} $a_i^*,b_i^*,c_i^*\in\mathbb{R}$, $i=0,1,\dots,d$, such that $b_d^*=c_0^*=0$, $b_{i-1}^*c_i^*\ne 0$, $i=1,2,\dots,d$, and that
\begin{gather}\label{*3-term recurrence}
	E_1\circ E_i=\frac{1}{|X|}(b_{i-1}^*E_{i-1}+a_i^*E_i+c_{i+1}^*E_{i+1}), \qquad i=0,1,\dots,d,
\end{gather}
where $\circ$ denotes the entrywise (or \emph{Hadamard} or \emph{Schur}) product of matrices, and $b_{-1}^*E_{-1}=c_{d+1}^*E_{d+1}:=0$.
Note that this property is dual to \eqref{3-term recurrence}.
See \cite[Section~5]{DKT2016EJC} for recent updates on the study of $Q$-polynomial distance-regular graphs.

From \eqref{3-term recurrence} it follows that there exist polynomials $v_i(\xi)\in\mathbb{R}[\xi]$, $i=0,1,\dots,d$, such that $\deg v_i(\xi)=i$ and $A_i=v_i(A)$.
Set $u_i(\xi)=v_i(\xi)/k_i$, $i=0,1,\dots,d$.
In general, by Leonard's theorem (see~\cite{Leonard1982SIAM}, \cite[Section~III.5]{BI1984B}), the polynomials $u_i$ associated with every $Q$-polynomial distance-regular graph are expressed in terms of the $q$-\emph{Racah polynomials} (cf.~\cite[Section~3.2]{KS1998R}) and their special/limit cases in the \emph{Askey scheme} of (basic) hypergeometric orthogonal polynomials~\mbox{\cite{KLS2010B,KS1998R}}.
In the most general (i.e., $q$-Racah) case, the $u_i$ are of the form
\begin{gather*}
u_i(\theta_j)=\basichypergeometricseries{4}{3}{\b^{-i},\,s^*\b^{i+1},\,\b^{-j},\,s\b^{j+1}}
{r_1\b,\,r_2\b,\,\b^{-d}}{\b;\b}\!, \qquad i,j=0,1,\dots,d,
\end{gather*}
where the parameters $r_1$, $r_2,s$, and $s^*$ satisfy $r_1r_2=ss^*\b^{d+1}\ne 0$, and we are using the standard notation for a basic hypergeometric series $_m\phi_n$: 
\begin{gather*}
\basichypergeometricseries{m}{n}{\mathfrak{a}_1,\dots,\mathfrak{a}_m} {\mathfrak{b}_1,\dots,\mathfrak{b}_n}{\b;\mathfrak{x}} =\sum_{h=0}^{\infty}\frac{(\mathfrak{a}_1;\b)_h\cdots(\mathfrak{a}_m;\b)_h} {(\mathfrak{b}_1;\b)_h\cdots(\mathfrak{b}_n;\b)_h}\, \frac{ (-1)^{(m-n-1)h} \mathfrak{x}^h }{ (\b;\b)_h \b^{(m-n-1)\binom{h}{2}} },
\end{gather*}
where $(\mathfrak{a};\b)_h$ denotes the $\b$-shifted factorial defined by
\begin{gather*}
	(\mathfrak{a};\b)_h=(1-\mathfrak{a})(1-\mathfrak{a}\b)\cdots\big(1-\mathfrak{a}\b^{h-1}\big), \qquad h=0,1,\dots.
\end{gather*}
To get the $u_i$ for our $\Gamma$, first fix $s,r_2\ne 0$ and let $s^*\rightarrow 0$ (so $r_1\rightarrow 0$), and then set
\begin{gather}\label{s, r_2 in terms of classical parameters}
	s=\frac{\alpha+1-\b}{(\alpha-\beta+\beta \b)\b^{d+1}}, \qquad r_2=\frac{\alpha}{(\alpha-\beta+\beta \b)\b},
\end{gather}
or equivalently,
\begin{gather*}
	\alpha=\frac{r_2(1-\b)}{s\b^d-r_2}, \qquad \beta=\frac{r_2\b-1}{\b(s\b^d-r_2)}.
\end{gather*}
See \cite[Proposition 6.2]{Tanaka2011EJC}.
The $u_i$ are the \emph{dual $q$-Hahn polynomials} (cf.~\cite[Section~3.7]{KS1998R}) for $s\ne 0$ and $r_2\ne 0$, the \emph{affine $q$-Krawtchouk polynomials} (cf.~\cite[Section~3.16]{KS1998R}) for $s=0$ and $r_2\ne 0$, and the \emph{dual $q$-Krawtchouk polynomials} (cf.~\cite[Section~3.17]{KS1998R}) for $s\ne 0$ and $r_2=0$.
See also~\cite[Examples 5.3--5.9]{Terwilliger2005DCC}.
We note that, in \cite[Proposition 6.2]{Tanaka2011EJC}, there is mentioned another case, called IA, which also corresponds to classical parameters with $\b\ne 1$.
The $u_i$ are then the \emph{quantum $q$-Krawtchouk polynomials} (cf.~\cite[Section~3.14]{KS1998R}).
However, it is known (see~\cite[Proposition~5.8]{DKT2016EJC}) that there exists no actual $\Gamma$ in this case.

For later use, we now establish another basic hypergeometric expression for the polynomials~$u_i$.
For the moment, fix $i,j=0,1,\dots,d$.
Recall Sear's transformation formula for a~terminating balanced $_4\phi_3$ series (cf.~\cite[Section~0.6]{KS1998R}):
\begin{gather*}
	\basichypergeometricseries{4}{3}{\b^{-i},\,\mathfrak{x},\,\mathfrak{y},\,\mathfrak{z}}{\mathfrak{l},\,\mathfrak{m},\,\mathfrak{n}}{\b;\b}=\basichypergeometricseries{4}{3}{\b^{-i},\,\mathfrak{x},\,\mathfrak{l}/\mathfrak{y},\,\mathfrak{l}/\mathfrak{z}}{\mathfrak{l},\,\mathfrak{x}\b^{1-i}/\mathfrak{m},\,\mathfrak{x}\b^{1-i}/\mathfrak{n}}{\b;\b} \frac{(\mathfrak{m}/\mathfrak{x};\b)_i(\mathfrak{n}/\mathfrak{x};\b)_i}{(\mathfrak{m};\b)_i(\mathfrak{n};\b)_i}\mathfrak{x}^i,
\end{gather*}
where $\mathfrak{xyz}\b^{1-i}=\mathfrak{lmn}\ne 0$.
Applying this formula twice and then simplifying a bit, we obtain
\begin{gather*}
	\basichypergeometricseries{4}{3}{\b^{-i},\,\mathfrak{x},\,\mathfrak{y},\,\mathfrak{z}}{\mathfrak{l},\,\mathfrak{m},\,\mathfrak{n}}{\b;\b} = \basichypergeometricseries{4}{3}{\b^{-i},\,\mathfrak{x},\,\mathfrak{xy}\b^{1-i}/\mathfrak{ln},\,\mathfrak{xz}\b^{1-i}/\mathfrak{ln}}{\mathfrak{x}\b^{1-i}/\mathfrak{n},\,\mathfrak{x}\b^{1-i}/\mathfrak{l},\,\mathfrak{m}}{\b;\b} \frac{(\mathfrak{n}/\mathfrak{x};\b)_i(\mathfrak{l}/\mathfrak{x};\b)_i}{(\mathfrak{n};\b)_i(\mathfrak{l};\b)_i}\mathfrak{x}^i.
\end{gather*}
Set
\begin{gather*}
	\mathfrak{x}=\b^{-j}, \qquad \mathfrak{y}=s^*\b^{i+1}, \qquad \mathfrak{z}=s\b^{j+1}, \qquad \mathfrak{l}=\b^{-d}, \qquad \mathfrak{m}=r_1\b, \qquad \mathfrak{n}=r_2\b
\end{gather*}
in this result.
Then we obtain the following expression for the $u_i$ for the $q$-Racah case:
\begin{gather*}
	u_i(\theta_j)=\basichypergeometricseries{4}{3}{\b^{-i},\,\b^{-j},\,s^*\b^{d-j+1}/r_2,\, s\b^{d-i+1}/r_2}{\b^{-i-j}/r_2,\,\b^{d-i-j+1},\,r_1\b}{\b;\b} \frac{(r_2\b^{j+1};\b)_i(\b^{j-d};\b)_i}{(r_2\b;\b)_i(\b^{-d};\b)_i\b^{ij}}.
\end{gather*}
By letting $s^*\rightarrow 0$, the RHS becomes
\begin{gather*}
\basichypergeometricseries{3}{2}{\b^{-i},\,\b^{-j},\,s\b^{d-i+1}/r_2}{\b^{-i-j}/r_2,\,\b^{d-i-j+1}}{\b;\b} \frac{(r_2\b^{j+1};\b)_i(\b^{j-d};\b)_i}{(r_2\b;\b)_i(\b^{-d};\b)_i\b^{ij}}
\\ \qquad
{}= \sum_{h=0}^i \frac{(\b^{-i};\b)_h (\b^{-j};\b)_h (r_2;s\b^{d-i+1};\b)_h \b^h}{(r_2;\b^{-i-j};\b)_h (\b^{d-i-j+1};\b)_h (\b;\b)_h} \, \frac{(r_2\b^{j+1};\b)_i(\b^{j-d};\b)_i}{(r_2\b;\b)_i(\b^{-d};\b)_i\b^{ij}},
\end{gather*}
where we write
\begin{gather*}
(\mathfrak{x};\mathfrak{y};\b)_h = (\mathfrak{x}-\mathfrak{y})(\mathfrak{x}-\mathfrak{y}\b)\cdots\big(\mathfrak{x}-\mathfrak{y}\b^{h-1}\big), \qquad h=0,1,\dots
\end{gather*}
for convenience.
We then set $s$ and $r_2$ as in \eqref{s, r_2 in terms of classical parameters}, and obtain the $u_i$ for $\Gamma$ as follows:
\begin{gather*}
	u_i(\theta_j) = \sum_{h=0}^i \frac{(\b^{-i};\b)_h (\b^{-j};\b)_h (\alpha;(\alpha+1-\b)\b^{1-i};\b)_h \b^h}{ (\alpha;(\alpha-\beta+\beta \b)\b^{1-i-j};\b)_h (\b^{d-i-j+1};\b)_h (\b;\b)_h} \frac{ (\alpha\!-\beta+\beta \b;\alpha \b^j;\b)_i (\b^{j-d};\b)_i }{ (\alpha\!-\beta+\beta \b;\alpha;\b)_i (\b^{-d};\b)_i \b^{ij}}.
\end{gather*}
Using \eqref{k_i} and \eqref{b and c} we have
\begin{gather*}
	k_i = \frac{(\b^{-d};\b)_i (\alpha-\beta+\beta \b;\alpha;\b)_i \b^{di}}{ (\alpha+1-\b;\alpha;\b)_i (\b;\b)_i },
\end{gather*}
from which it follows that
\begin{gather}
	v_i(\theta_j) = k_i u_i(\theta_j) = \sum_{h=0}^i \frac{(\b^{-j};\b)_h (\b^{j-d};\b)_{i-h} (\alpha-\beta+\beta \b;\alpha \b^j;\b)_{i-h}\b^{(i-h)(d-j)+jh}}{(\b;\b)_h (\alpha+1-\b;\alpha;\b)_{i-h} (\b;\b)_{i-h}}. \label{v_i}
\end{gather}

Let $m_i=\operatorname{tr}(E_i)$, the multiplicity of $\theta_i$ in the spectrum of $\Gamma$.
This value is computed in \cite[Theorem 8.4.3]{BCN1989B}:
\begin{align}
	m_i={}&\frac{ \prod_{h=0}^{i-1}\gauss{d-h}{1}\bigl(\beta-\gauss{h}{1}\alpha\bigr)
\bigl(1+\gauss{d-h}{1}\alpha+\b^{d-h}\beta\bigr) }{ \prod_{h=1}^i\gauss{h}{1}\bigl(\beta-\gauss{h}{1}\alpha+\b^h\bigr)\bigl(1+\gauss{d-h}{1}\alpha\bigr) } \, \frac{ \bigl(1+\gauss{d-2i}{1}\alpha+\b^{d-2i}\beta\bigr)\b^i }{ 1+\gauss{d}{1}\alpha+\b^d\beta } \notag \\
={}&\frac{ (\b^{-d};\b)_i (\alpha-\beta+\beta\b;\alpha;\b)_i (\alpha-\beta+\beta\b;(\alpha+1-\b)\b^{-d};\b)_i }{ (\b;\b)_i (\alpha-\beta+\beta\b;(\alpha+1-\b)\b;\b)_i (\alpha+1-\b;\alpha\b^{d-i};\b)_i } \notag \\
	& \times \frac{ (\alpha-\beta+\beta\b-(\alpha+1-\b)\b^{2i-d}) \b^{2di-i^2} }{\alpha-\beta+\beta\b-(\alpha+1-\b)\b^{-d}}. \label{m_i}
\end{align}

Finally, we obtain a closed formula for $|X|$, the number of vertices of $\Gamma$.
Recall the $\mathbb{C}$-vector space $W(\Gamma)$ from \eqref{primary module}.
In view of \eqref{second basis}, $W(\Gamma)$ has another orthonormal basis $\Psi_0,\Psi_1,\dots,\Psi_d$ defined by
\begin{gather*}
	\Psi_i=\sqrt{\frac{|X|}{m_i}}E_i\hat{o}, \qquad i=0,1,\dots,d.
\end{gather*}
As in the proof of Proposition \ref{negative powers of q}, write
\begin{gather*}
	E_1=\frac{1}{|X|}\sum_{i=0}^d \theta_i^*A_i.
\end{gather*}
Now, let
\begin{gather*}
	D=|X|\operatorname{diag} E_1\hat{o}.
\end{gather*}
Then we have
\begin{gather*}
	A\Psi_i=\theta_i\Psi_i, \qquad D\Phi_i=\theta_i^*\Phi_i, \qquad i=0,1,\dots,d.
\end{gather*}
Moreover, it follows from \eqref{3-term recurrence} and \eqref{*3-term recurrence} that the matrix representing the action of $A$ (resp.~$D$) on $W(\Gamma)$ with respect to the $\Phi_i$ (resp.~the $\Psi_i$) is tridiagonal with non-zero superdiagonal and subdiagonal entries.
This means that $A$ and $D$ act on $W(\Gamma)$ as a \emph{Leonard pair} in the sense of \cite[Definition~1.1]{Terwilliger2001LAA}.
In the theory of Leonard pairs, there is a scalar denoted by $\nu$ \cite[Definition~9.3]{Terwilliger2004LAA}, and it is easy to see that $\nu=\langle \Phi_0,\Psi_0\rangle^{-2}=|X|$ for the above Leonard pair on $W(\Gamma)$.
For the $q$-Racah case, the scalar $\nu$ is given in \cite[p.~273]{Terwilliger2004LAA} as follows:
\begin{gather*}
	\nu=\frac{ \big(s\b^2;\b\big)_d \big(s^*\b^2;\b\big)_d }{ r_1^d \b^d (s\b/r_1;\b)_d (s^*\b/r_1;\b)_d }.
\end{gather*}
Again by letting $s^*\rightarrow 0$ and then setting $s$ and $r_2$ as in \eqref{s, r_2 in terms of classical parameters}, it follows that
\begin{gather}\label{|X|}
	|X|=\frac{ (-1)^d (\alpha-\beta+\beta\b;(\alpha+1-\b)\b^{1-d};\b)_d \b^{\binom{d}{2}} }{ (\alpha+1-\b;\alpha;\b)_d }.
\end{gather}

\section{Description of asymptotic normalized spectral distributions}\label{sec: description of limit distributions}

In this section, we describe the asymptotic normalized spectral distributions (cf.~Remark \ref{limit distribution}) corresponding to Theorem \ref{classification}\,$(i)$ and $(ii)$, following~\cite{Hora2000PTRF}.

We retain the notation of the previous section.
The Borel probability measure $\mu$ on $\mathbb{R}$ from~\eqref{distribution} associated with the normalized adjacency matrix \eqref{normalized variable} is given by
\begin{gather}\label{normalized distribution}
	\mu\bigg(\frac{\theta_j-\q k}{\Sigma_{\q}(A)}\bigg)=\sum_{i=0}^d \q^i v_i(\theta_j)\, \frac{m_j}{|X|}, \qquad j=0,1,\dots,d,
\end{gather}
which follows from \eqref{Gibbs state as trace} and since
\begin{gather*}
	\bigg(\frac{A-\q kI}{\Sigma_{\q}(A)}\bigg)^{\!\!\ell} \! E_j=\bigg(\frac{\theta_j-\q k}{\Sigma_{\q}(A)}\bigg)^{\!\!\ell} \! E_j, \qquad A_i E_j=v_i(\theta_j) E_j
\end{gather*}
for $\ell=0,1,\dots$ and $i,j=0,1,\dots,d$.
From \eqref{m_i} and \eqref{|X|} it follows that
\begin{align}
	\frac{m_j}{|X|}= {}&\frac{ (\b^{-d};\b)_j (\alpha-\beta+\beta\b;\alpha;\b)_j (\alpha+1-\b;\alpha;\b)_{d-j} }{ (\b;\b)_j (\alpha-\beta+\beta\b;(\alpha+1-\b)\b^{j-d};\b)_{d+1} } \notag \\
	& \times (-1)^d(\alpha-\beta+\beta\b-(\alpha+1-\b)\b^{2j-d}) \b^{2dj-j^2-\binom{d}{2}} \label{m_i/|X|}
\end{align}
for $j=0,1,\dots,d$.
From \eqref{v_i} it follows that
\begin{align}
	\sum_{i=0}^d \q^i v_i(\theta_j) &= \sum_{h=0}^d \frac{(\b^{-j};\b)_h\b^{jh} \q^h}{(\b;\b)_h} \notag \,\sum_{i=h}^d \frac{(\b^{j-d};\b)_{i-h} (\alpha-\beta+\beta \b;\alpha \b^j;\b)_{i-h} \b^{(i-h)(d-j)}\q^{i-h} }{(\alpha+1-\b;\alpha;\b)_{i-h} (\b;\b)_{i-h}} \notag \\
	&= \basichypergeometricseries{1}{0}{\b^{-j}}{-}{\b;\b^j\q} \sum_{\ell=0}^{d-j} \frac{(\b^{j-d};\b)_{\ell} (\alpha-\beta+\beta \b;\alpha \b^j;\b)_{\ell} \b^{\ell(d-j)}\q^{\ell} }{ (\alpha+1-\b;\alpha;\b)_{\ell} (\b;\b)_{\ell}} \notag \\
	&= (\q;\b)_j \sum_{\ell=0}^{d-j} \frac{(\b^{j-d};\b)_{\ell} (\alpha-\beta+\beta \b;\alpha \b^j;\b)_{\ell} \b^{\ell(d-j)}\q^{\ell} }{ (\alpha+1-\b;\alpha;\b)_{\ell} (\b;\b)_{\ell}} \label{sum t^i v_i}
\end{align}
for $j=0,1,\dots,d$, where we have used the $q$-binomial theorem (cf.~\cite[Section~0.5]{KS1998R})
\begin{gather}\label{q-binomial theorem}
	\basichypergeometricseries{1}{0}{\b^{-n}}{-}{\b;\mathfrak{x}} = (\mathfrak{x} \b^{-n};\b)_n, \qquad n=0,1,2,\dots.
\end{gather}
Note that the last sum in \eqref{sum t^i v_i} is a $_2\phi_1$ in general.

\subsection[Case beta/sqrt k rightarrow rho>0]
{Case $\boldsymbol{\beta/\sqrt{k}\rightarrow\rho>0}$}
\label{subsec: 1st case}

With reference to Assumption \ref{assume classical parameters}, suppose that we are in Theorem \ref{classification}\,$(i)$, or $(ii)$ with $\rho>0$.
For $(i)$, we moreover assume that $\rho$ is indeed an accumulation point of $\beta/\sqrt{k}$, and will consider a subnet of $(\Gamma_{\lambda})_{\lambda\in\Lambda}$ for which $\beta/\sqrt{k}\rightarrow\rho$ if necessary (i.e., if $\alpha/\rho$ is also an accumulation point).
We note that $k\rightarrow\infty$ (cf.~\eqref{k diverges}), $|\beta|\rightarrow\infty$, $t\rightarrow 0$, and that
\begin{gather*}
	\frac{\beta(\b-1)}{q^d}\sim \frac{\beta^2}{k}\rightarrow \rho^2, \qquad \q\beta=t\sqrt{k}\frac{\beta}{\sqrt{k}}\rightarrow \gamma\rho, \qquad \q\b^d\rightarrow \frac{\gamma(q-1)}{\rho}.
\end{gather*}
Using this, \eqref{eigenvalues}, \eqref{m_i/|X|}, and \eqref{sum t^i v_i}, we routinely compute
{\samepage\begin{gather}
	\frac{\theta_{d-j}-\q k}{\Sigma_{\q}(A)} \rightarrow \frac{ \gauss{j}{1}(\rho-\alpha/\rho\b^j) -1/\rho\b^j -\gamma }{ \sqrt{1+\gamma(\rho+\alpha/\rho)} }, \label{case 1: limit points}
\\
	\frac{m_{d-j}}{|X|} \rightarrow \frac{ \big(\alpha/\rho^2\b^{j+1};\b^{-1}\big)_{\infty} (\alpha+1-\b;\alpha;\b)_j \big(1-(\alpha+1-\b)/\rho^2\b^{2j}\big) }{ \big((\alpha+1-\b)/\rho^2\b^j;\b^{-1}\big)_{\infty} (\b;\b)_j \rho^{2j} \b^{j^2-j} }, \label{case 1: limit masses 1st part}
\\
	\sum_{i=0}^d \q^i v_i(\theta_{d-j}) \rightarrow \big(\gamma(\b-1)/\rho\b^{j+1};\b^{-1}\big)_{\infty} \sum_{\ell=0}^j \frac{ (\b^{-j};\b)_{\ell} \big(\alpha/\rho^2\b^j;\b\big)_{\ell} \gamma^{\ell} \rho^{\ell} \b^{j\ell} (\b-1)^{\ell} }{ (\alpha+1-\b;\alpha;\b)_{\ell} (\b;\b)_{\ell} } \label{case 1: limit masses 2nd part}
\end{gather}}\noindent
for $j=0,1,2,\dots$.
The measure \eqref{normalized distribution} converges weakly to the discrete measure $\mu_{\infty}$ on $\mathbb{R}$ defined on the limit points in \eqref{case 1: limit points}, where the masses are given by the products of the limits in \eqref{case 1: limit masses 1st part} and \eqref{case 1: limit masses 2nd part}.\footnote{This follows for example from the observation that $\mu((a,b))\rightarrow\mu_{\infty}((a,b))$ for every bounded open interval $(a,b)$ in $\mathbb{R}$ and \cite[Theorem 8.2.17]{Bogachev2007B}. When $q>0$, it is also immediate to check that the distribution function of $\mu$ converges to that of $\mu_{\infty}$ at the points of continuity of the latter. }

\subsection[Case alpha not= 0, beta sqrt k rightarrow alpha/rho]
{Case $\boldsymbol{\alpha\ne 0}$, $\boldsymbol{\beta/\sqrt{k}\rightarrow \alpha/\rho}$}
\label{subsec: 2nd case}

With reference to Assumption \ref{assume classical parameters}, suppose that we are in Theorem \ref{classification}\,$(i)$, and that $\alpha/\rho$ is an accumulation point of $\beta/\sqrt{k}$.
We will consider a subnet of $(\Gamma_{\lambda})_{\lambda\in\Lambda}$ for which $\beta/\sqrt{k}\rightarrow\alpha/\rho$ if necessary.
The formulas for the limit distribution are simply obtained by replacing $\rho$ by $\alpha/\rho$ in those of the previous case:
\begin{gather*}
	\frac{\theta_{d-j}-\q k}{\Sigma_{\q}(A)} \rightarrow \frac{ \gauss{j}{1}\big(\alpha/\rho-\rho/\b^j\big) -\rho/\alpha\b^j -\gamma }{ \sqrt{1+\gamma(\rho+\alpha/\rho)} },
\\
	\frac{m_{d-j}}{|X|} \rightarrow \frac{ \big(\rho^2/\alpha\b^{j+1};\b^{-1}\big)_{\infty} (\alpha+1-\b;\alpha;\b)_j \big(1-(\alpha+1-\b)\rho^2/\alpha^2\b^{2j}\big) \rho^{2j} }{ \big((\alpha+1-\b)\rho^2/\alpha^2\b^j;\b^{-1}\big)_{\infty} (\b;\b)_j \alpha^{2j} \b^{j^2-j} },
\\
	\sum_{i=0}^d \q^i v_i(\theta_{d-j}) \rightarrow \big(\gamma\rho(\b-1)/\alpha\b^{j+1};\b^{-1}\big)_{\infty} \sum_{\ell=0}^j \frac{ (\b^{-j};\b)_{\ell} \big(\rho^2/\alpha\b^j;\b\big)_{\ell} \alpha^{\ell} \gamma^{\ell} \b^{j\ell} (\b-1)^{\ell} }{ (\alpha+1-\b;\alpha;\b)_{\ell} (\b;\b)_{\ell} \rho^{\ell} }
\end{gather*}
for $j=0,1,2,\dots$.
We note that, while the roles of $\rho$ and $\alpha/\rho$ are interchangeable when $\b>0$, their distinction is essential when $\b<0$, as $\rho=\sqrt{-\alpha}$ and $\alpha/\rho=-\sqrt{-\alpha}$.

\subsection[Case alpha = 0, beta sqrt k rightarrow 0]
{Case $\boldsymbol{\alpha=0}$, $\boldsymbol{\beta/\sqrt{k}\rightarrow 0}$}
\label{subsec: 3rd case}

With reference to Assumption \ref{assume classical parameters}, suppose that we are in Theorem \ref{classification}\,$(ii)$ with $\rho=0$.
Note that $\b>0$ in this case, and let
\begin{gather*}
	c=\bigl\lfloor \log_{\b}\sqrt{k}\bigr\rfloor.
\end{gather*}
Then we have $\sqrt{k}/\b^c\in[1,\b)$.
Let $\eta/\sqrt{\b-1}\in[1,\b]$ be an accumulation point of $\sqrt{k}/\b^c$, and consider a subnet of $(\Gamma_{\lambda})_{\lambda\in\Lambda}$ for which $\sqrt{k}/\b^c\rightarrow\eta/\sqrt{\b-1}$ if $\eta$ is not unique.
We note that $k\rightarrow\infty$, $\q\rightarrow 0$, $c\rightarrow\infty$, $d-c\rightarrow\infty$, and that
\begin{gather*}
	\beta\b^{d-2c} \sim \frac{k(\b-1)}{\b^{2c}}\rightarrow\eta^2, \qquad \q\beta\rightarrow 0, \qquad \q\b^c\rightarrow \frac{\gamma\sqrt{\b-1}}{\eta}.
\end{gather*}
Using this, \eqref{eigenvalues}, \eqref{m_i/|X|}, and \eqref{sum t^i v_i}, we obtain
\begin{gather*}
\frac{\theta_{c-j}-\q k}{\Sigma_{\q}(A)} \rightarrow \frac{ \eta\b^j-1/\eta\b^j }{ \sqrt{q-1} } -\gamma,
\\
\frac{m_{c-j}}{|X|} \rightarrow \frac{ \big(1+1/\eta^2\b^{2j}\big) \b^{-2j^2+j} }{ \big(\b^{-1};\b^{-1}\big)_{\infty} \big({-}1/\eta^2;\b^{-1}\big)_{\infty} \big({-}\eta^2/\b;\b^{-1}\big)_{\infty} \eta^{4j} },
\\
\sum_{i=0}^d \q^i v_i(\theta_{c-j}) \rightarrow \big(\gamma\sqrt{\b-1}/\eta\b^{j+1};\b^{-1}\big)_{\infty} \big({-}\gamma\eta\b^{j-1}\sqrt{\b-1};\b^{-1}\big)_{\infty}
\end{gather*}
for $j=0,\pm 1,\pm 2,\dots$, where we have used
\begin{gather*}
	\big({-}\b^{c-j-d}/\beta;\b\big)_{d+1} = \big({-}\b^{c-j-d}/\beta;\b\big)_{c+j+1} \big({-}\b^{2c-d+1}/\beta;\b\big)_{d-c-j}
\end{gather*}
for the second formula, and \eqref{q-binomial theorem} for the third one.
Again, the measure \eqref{normalized distribution} converges weakly to the discrete measure $\mu_{\infty}$ on $\mathbb{R}$ defined by the above limits.

\section{Examples}\label{sec: examples}

There are currently eleven known infinite families of distance-regular graphs having classical parameters with unbounded diameter and such that $\b\ne 1$.
In this section, we apply the results of the previous sections to these eleven families.
For more detailed information on these families, see the references given.
It should be remarked that, by virtue of Proposition~\ref{negative powers of q}, there exist infinitely many choices for the scalar $\gamma$ in Theorem~\ref{classification}\,$(i)$ and $(ii)$.

\subsection{Grassmann graphs and twisted Grassmann graphs}\label{subsec: Grassmann}

The \emph{Grassmann graph} $J_{\b}(n,d)$ has as vertices the $d$-dimensional subspaces of the $n$-dimensional vector space $\mathbb{F}_{\b}^n$ over the finite field $\mathbb{F}_{\b}$ with $\b$ elements, where two vertices $x$ and $y$ are adjacent when $\dim x\cap y=d-1$; cf.~\cite[Section~9.1]{BCN1989B}.
We always assume that $n\geqslant 2d$, as $J_{\b}(n,d)$ and $J_{\b}(n,n-d)$ are isomorphic.
The graph $J_{\b}(n,d)$ has classical parameters $(d,\b,\alpha,\beta)$, where
\begin{gather*}
	\alpha=\b, \qquad \beta=\b\genfrac{[}{]}{0pt}{}{n-d}{1}.
\end{gather*}
Fix $\b$ and let $d\rightarrow\infty$, $t\sqrt{k}\rightarrow \gamma$, and let also $n-2d+1\rightarrow 2\delta$ for some $\delta\in \frac{1}{2}\mathbb{Z}$, so that we have $\beta/\sqrt{k}\rightarrow \rho:=\b^{\delta}$.
We are in Theorem \ref{classification}\,$(i)$, and from the results of Section~\ref{subsec: 1st case} it follows that the measure \eqref{normalized distribution} converges weakly to $\mu_{\infty}$, where
\begin{gather*}
\mu_{\infty}\bigg(\frac{\b^{\delta+j}+\b^{-\delta-j}-\b^{\delta}-\b^{1-\delta}-\gamma(\b-1)}
{(\b-1)\sqrt{1+\gamma(\b^{\delta}+\b^{1-\delta})}}\bigg)
\\ \qquad
{} = \bigg(\frac{1}{\b^{j(2\delta+j-1)}}-\frac{1}{\b^{(j+1)(2\delta+j)}}\bigg) \big(\gamma(\b-1)/\b^{\delta+j+1};\b^{-1}\big)_{\infty}
\\ \qquad\hphantom{=}
{}\times \basichypergeometricseries{2}{1}{\b^{-j},\b^{1-2\delta-j}}{\b}{\b;\gamma(\b-1)\b^{\delta+j}}
\end{gather*}
for $j=0,1,2,\dots$.
Note that $\delta\geqslant 1/2$ and $\alpha/\rho=\b^{1-\delta}$, from which it follows that a different choice of $\delta$ gives rise to a different limit in Theorem \ref{classification}\,$(i)$.
Hora \cite{Hora1998IDAQPRT} previously obtained $\mu_{\infty}$ for the vacuum state $\varphi_0$, i.e., for $\gamma=0$.

The \emph{twisted Grassmann graph} $\tilde{J}_{\b}(2d+1,d)$ is defined as follows.
Fix a hyperplane $H$ of $\mathbb{F}_{\b}^{2d+1}$.
The vertex set consists of the $(d+1)$-dimensional subspaces of $\mathbb{F}_{\b}^{2d+1}$ which are not contained in $H$, together with the $(d-1)$-dimensional subspaces of $H$.
Two vertices $x$ and $y$ are adjacent when $\dim x+\dim y-2\dim x\cap y=2$.
See \cite{DK2005IM}.
The graph $\tilde{J}_{\b}(2d+1,d)$ has the same classical parameters as $J_{\b}(2d+1,d)$, and hence we obtain the above measure $\mu_{\infty}$ with $\delta=1$.

\subsection{Dual polar graphs and Hemmeter graphs}
\label{subsec: dual polar}

Consider one of the following vector spaces $V$ endowed with a non-degenerate form:
\begin{quote}
$C_d(\b)$: $V=\mathbb{F}_{\b}^{2d}$ with a symplectic form; \\[.5ex]
$B_d(\b)$: $V=\mathbb{F}_{\b}^{2d+1}$ with a quadratic form; \\[.5ex]
$D_d(\b)$: $V=\mathbb{F}_{\b}^{2d}$ with a quadratic form of Witt index $d$; \\[.5ex]
$^2\!D_{d+1}(\b)$: $V=\mathbb{F}_{\b}^{2d+2}$ with a quadratic form of Witt index $d$; \\[.5ex]
$^2\!A_{2d}(r)$: $V=\mathbb{F}_{\b}^{2d+1}$ with a Hermitian form ($\b=r^2$); \\[.5ex]
$^2\!A_{2d-1}(r)$: $V=\mathbb{F}_{\b}^{2d}$ with a Hermitian form ($\b=r^2$).
\end{quote}
We note that maximal isotropic subspaces of $V$ have dimension $d$.
The \emph{dual polar graph} on $V$ has as vertices these maximal isotropic subspaces, where two vertices $x$ and $y$ are adjacent when $\dim x\cap y=d-1$; cf.~\cite[Section~9.4]{BCN1989B}.
This graph has classical parameters $(d,\b,0,\b^e)$, where we let $e$ be 1, 1, 0, 2, 3/2, 1/2 in the respective types $C_d(\b)$, $B_d(\b)$, $D_d(\b)$, $^2\!D_{d+1}(\b)$, $^2\!A_{2d}(r)$, and~$^2\!A_{2d-1}(r)$.
Fix one of these types as well as $\b$, and let $d\rightarrow\infty$, $t\sqrt{k}\rightarrow \gamma$.
We~have $\beta/\sqrt{k}\rightarrow 0$, and hence we are in Theorem \ref{classification}\,$(ii)$ with $\rho=0$.
Let the scalar $c$ be as in Section~\ref{subsec: 3rd case}.
Note that
\begin{gather*}
	c=\bigg\lfloor \frac{\log_{\b}(\b^d-1)-\log_{\b}(\b-1)+e}{2}\bigg\rfloor.
\end{gather*}
Using $d\sim \log_{\b}(\b^d-1)<d$ and $0\leqslant\log_{\b}(\b-1)<1$, we obtain the value of $c$ for sufficiently large~$d$ as follows:
\begin{center}\setlength{\tabcolsep}{5.5pt}
\renewcommand{\arraystretch}{1.2}
\begin{tabular}{c|c|c|c|c|c}
\hline
& $e=1$ & $e=0$ & $e=2$ & $e=3/2$ & $e=1/2$ \\
\hline
$d$ even & $d/2$ & $d/2-1$ & $d/2$ & $d/2$ & $d/2-1$ \\
$d$ odd & $(d-1)/2$ & $(d-1)/2$ & $(d+1)/2$ & $(d-1)/2$ & $(d-1)/2$ \\
\hline
\end{tabular}
\end{center}
For the last two cases ($e\in\{3/2,1/2\}$), we have also used $\b\geqslant 4$ (as $\b=r^2$ is a square) and $\log_4 3=0.792\dots$.
It follows that $\sqrt{k}/\b^c\in[1,\b)$ has two accumulation points, and considering limits for even $d$ and odd $d$ separately, the scalar $\eta$ from Section~\ref{subsec: 3rd case} is given as in the following table:
\begin{center}\setlength{\tabcolsep}{7.5pt}
\renewcommand{\arraystretch}{1.2}
\begin{tabular}{c|c|c|c|c|c}
\hline
& $e=1 $& $e=0$ & $e=2$ & $e=3/2$ & $e=1/2$ \\
\hline
$d$ even & $\b^{1/2}$ & $\b$ & $\b$ & $\b^{3/4}$ & $\b^{5/4}$ \\
$d$ odd & $\b$ & $\b^{1/2}$ & $\b^{1/2}$ & $\b^{5/4}$ & $\b^{3/4}$ \\
\hline
\end{tabular}
\end{center}
The measure \eqref{normalized distribution} converges weakly to the measure $\mu_{\infty}$ as described in Section~\ref{subsec: 3rd case}.
(We do not write down the result here, since there are four values of $\eta$ and also since the substitution of these values does not seem to simplify the formula significantly.)

Note that the graphs $C_d(\b)$ and $B_d(\b)$ share the same classical parameters $(d,\b,0,\b)$.
The \emph{extended bipartite double} of a graph $\Gamma=(X,R)$ is the graph with vertex set $\mathbb{F}_2\times X$, where a~vertex $(\epsilon,x)$ is adjacent to $(\epsilon+1,x)$ and all the vertices $(\epsilon+1,y)$ with $x\sim y$.
The graph $D_d(\b)$ is shown to be isomorphic to the extended bipartite double of $B_{d-1}(\b)$.
The \emph{Hemmeter graph} $\operatorname{Hem}_d(\b)$ is then defined as the extended bipartite double of $C_{d-1}(\b)$; cf.~\cite[Section~9.4C]{BCN1989B}.
It has the same classical parameters $(d,\b,0,1)$ as $D_d(\b)$, so that we obtain the above $\mu_{\infty}$ for $e=0$.

\subsection{Half dual polar graphs and Ustimenko graphs}
\label{subsec: half dual polar}

Recall that a graph $\Gamma=(X,R)$ is said to be \emph{bipartite} whenever there is a bipartition $X=X^+\sqcup X^-$ such that no edge is contained in $X^+$ or $X^-$.
In this case, a \emph{halved graph} of $\Gamma$ has as vertex set either $X^+$ or $X^-$, where two distinct vertices are adjacent when there is a~path of length two joining them in $\Gamma$.
The dual polar graph $D_n(r)$ and the Hemmeter graph $\operatorname{Hem}_n(r)$ are bipartite, and their halved graphs are called the \emph{half dual polar graph} $D_{n,n}(r)$ and the \emph{Ustimenko graph} $\operatorname{Ust}_n(r)$, respectively; cf.~\cite[Section~9.4C]{BCN1989B}.
These graphs have classical parameters $(d,\b,\alpha,\beta)$, where
\begin{gather*}
	d=\left\lfloor \frac{n}{2}\right\rfloor, \qquad \b=r^2, \qquad \alpha=r(r+1), \qquad \beta=\frac{r(r^m-1)}{r-1},
\end{gather*}
where $m=2\lceil n/2\rceil-1$.
Fix $r$, and let $d\rightarrow\infty$, $t\sqrt{k}\rightarrow \gamma$.
Note that $\beta/\sqrt{k}$ has two accumulation points $\sqrt{r+1}$ and $r\sqrt{r+1}=\alpha/\sqrt{r+1}$, where $\beta/\sqrt{k}\rightarrow\sqrt{r+1}$ for even $n$, and $\beta/\sqrt{k}\rightarrow r\sqrt{r+1}$ for odd $n$.
We consider limits for even $n$ and odd $n$ separately.
Set $\epsilon=0$ for even $n$, and $\epsilon=1$ for odd $n$.
From the results of Sections \ref{subsec: 1st case} and \ref{subsec: 2nd case} it follows that the measure \eqref{normalized distribution} converges weakly to $\mu_{\infty}$, where
\begin{gather*}
\mu_{\infty}\bigg(\frac{r^{\epsilon+2j}+r^{-\epsilon-2j}-r-1-\gamma(r-1)\sqrt{r+1}}
{(r-1)\sqrt{r+1+\gamma(r+1)^{5/2}}}\bigg)
\\ \qquad
{}= \frac{ r^{2\epsilon+4j}-1 }{ r^{\epsilon+2j}-1 } \, \frac{ (r^{-1};r^{-2})_{\infty}}{r^{(\epsilon+j)(2j+1)}(r^{-2};r^{-2})_{\infty}} \, \big(\gamma(r-1)\sqrt{r+1}/r^{\epsilon+2j+2};r^{-2}\big)_{\infty}
\\ \qquad\hphantom{=}
{}\times \basichypergeometricseries{2}{1}{r^{-2j},r^{1-2\epsilon-2j}}{r}{r^2;\gamma r^{\epsilon+2j}(r-1)\sqrt{r+1}}
\end{gather*}
for $j=0,1,2,\dots$, where we understand that $\big(r^0-1\big)/\big(r^0-1\big)=1$ when $\epsilon=j=0$.

\subsection[Second classical parameters for dual polar graphs 2A2d-1(r)]
{Second classical parameters for dual polar graphs $\boldsymbol{{}^2\!A_{2d-1}(r)}$}\label{subsec: second dual polar}

The Hermitian dual polar graph $^2\!A_{2d-1}(r)$ has another set of classical parameters $(d,\b,\alpha,\beta)$, where (cf.~\eqref{two expressions})
\begin{gather*}
\b=-r, \qquad \alpha=\frac{r(r+1)}{1-r}, \qquad \beta=\frac{r(1+(-r)^d)}{1-r}.
\end{gather*}
Fix $r$, and let $d\rightarrow\infty$, $t\sqrt{k}\rightarrow \gamma$.
Note that $\beta/\sqrt{k}$ has two accumulation points $\pm\sqrt{-\alpha}$, where $\beta/\sqrt{k}\rightarrow -\sqrt{-\alpha}$ for even $d$, and $\beta/\sqrt{k}\rightarrow \sqrt{-\alpha}$ for odd $d$.
We consider limits for even $d$ and odd $d$ separately.
According to whether $\beta/\sqrt{k}\rightarrow \pm\sqrt{-\alpha}$, the measure \eqref{normalized distribution} converges weakly to $\mu_{\infty}$, where
\begin{gather*}
\mu_{\infty}\bigg( \frac{\mp\sqrt{r} (-r)^j \pm \sqrt{r}^{-1} (-r)^{-j} }{ \sqrt{r^2-1} } -\gamma \bigg)
\\ \qquad
{}= \frac{ (r^{2j+1}\!+\!1) (r^{-1};-r^{-1})_{\infty} }{ r^{(j+1)^2} (-r^{-1};-r^{-1})_{\infty} }
\bigg(\!\!\mp \gamma (-r)^{-j-1} \sqrt{\frac{r^2\!-\!1}{r}};-r^{-1}\! \bigg)_{\!\!\infty}\!\! \bigg(\!\! \mp \gamma (-r)^{j-1} \sqrt{\frac{r^2\!-\!1}{r}}; r^{-2}\!\bigg)_{\!\! j}
\end{gather*}
for $j=0,1,2,\dots$.
Here we have again used \eqref{q-binomial theorem} to get the result.
We may routinely verify that this measure is identical to the one in Section~\ref{subsec: dual polar} with $e=1/2$, using
\begin{gather*}
	\big(r^{-1};-r^{-1}\big)_{\infty} \big({-}r^{-1};r^{-2}\big)_{\infty} = 1,
\end{gather*}
which is a special case of Lebesgue's identity; see, e.g., \cite{Fu2008DM}.

\subsection{Sesquilinear forms graphs}\label{subsec: sesquilinear forms}

There are four infinite families of sesquilinear forms graphs, all of which are Cayley graphs.

The \emph{bilinear forms graph} $\operatorname{Bil}(d\times e,\b)$ has as vertices the $d\times e$ matrices over $\mathbb{F}_{\b}$, where two vertices $x$ and $y$ are adjacent when $\operatorname{rank}(x-y)=1$; cf.~\cite[Section~9.5A]{BCN1989B}.
We always assume that $d\leqslant e$, as $\operatorname{Bil}(d\times e,\b)$ and $\operatorname{Bil}(e\times d,\b)$ are isomorphic.
The graph $\operatorname{Bil}(d\times e,\b)$ has classical parameters $(d,\b,\b-1,\b^e-1)$.
Fix $\b$ and let $d\rightarrow\infty$, $t\sqrt{k}\rightarrow \gamma$, and $e-d\rightarrow 2\delta$ for some $\delta\in \frac{1}{2}\mathbb{Z}$, so that $\beta/\sqrt{k}\rightarrow \rho:=\b^{\delta}\sqrt{\b-1}$.
The measure~\eqref{normalized distribution} converges weakly to $\mu_{\infty}$, where
\begin{gather*}
\mu_{\infty}\bigg(\frac{\b^{\delta+j}-\b^{\delta}-\b^{-\delta}-\gamma\sqrt{\b-1}}{\sqrt{ \b-1+\gamma(\b^{\delta}+\b^{-\delta})(\b-1)^{3/2} }}\bigg)
\\ \qquad
{}= \frac{ (-1)^j (\b^{-2\delta-j-1};\b^{-1})_{\infty} }{ (\b;\b)_j \b^{2\delta j+\binom{j}{2}} } \, \bigl(\gamma\sqrt{\b\!-\!1}/\b^{\delta+j+1};\b^{-1}\bigr)_{\infty}
 \basichypergeometricseries{2}{0}{\b^{-j},\b^{-2\delta-j}}{-}{\b;\gamma\b^{\delta+j}\sqrt{\b\!-\!1}}
\end{gather*}
for $j=0,1,2,\dots$.
Since $\delta\geqslant 0$ and $\alpha/\rho=\b^{-\delta}\sqrt{\b-1}$, it follows that a different choice of $\delta$ gives rise to a different limit in Theorem \ref{classification}\,$(i)$.

The \emph{alternating forms graph} $\operatorname{Alt}(n,r)$ has as vertices the $n\times n$ skew-symmetric matrices with zero diagonal over $\mathbb{F}_r$, where two vertices $x$ and $y$ are adjacent when $\operatorname{rank}(x-y)=2$; cf.~\cite[Section~9.5B]{BCN1989B}.
It has classical parameters $(d,\b,\alpha,\beta)$, where
\begin{gather*}
d=\bigg\lfloor \frac{n}{2}\bigg\rfloor, \qquad \b=r^2, \qquad \alpha=r^2-1, \qquad \beta=r^m-1,
\end{gather*}
where $m=2\lceil n/2\rceil-1$.
Fix $r$, and let $d\rightarrow\infty$, $t\sqrt{k}\rightarrow \gamma$.
Note that we can apply the above computation with $e:=m/2$.
Hence we consider limits for even $n$ and odd $n$ separately, and set $\delta=-1/4$ for even $n$ and $\delta=1/4$ for odd $n$.
We then obtain the above measure $\mu_{\infty}$ with $\b=r^2$.

The \emph{quadratic forms graph} $\operatorname{Qua}(n-1,r)$ has as vertices the quadratic forms on $\mathbb{F}_r^{n-1}$, where two vertices $x$ and $y$ are adjacent when $\operatorname{rank}(x-y)\in\{1,2\}$.
See \cite[Section~9.6]{BCN1989B} for a precise description.
This graph has the same classical parameters as $\operatorname{Alt}(n,r)$, and thus the result is the same as above.

Finally, the \emph{Hermitian forms graph} $\operatorname{Her}\big(d,r^2\big)$ has as vertices the $d\times d$ Hermitian matrices over~$\mathbb{F}_{r^2}$, where two vertices $x$ and $y$ are adjacent when $\operatorname{rank}(x-y)=1$; cf.~\cite[Section~9.5C]{BCN1989B}.
It~has classical parameters $(d,-r,-r-1,-(-r)^d-1)$.
Fix $r$, and let $d\rightarrow\infty$, $t\sqrt{k}\rightarrow \gamma$.
We consider limits for even $d$ and odd $d$ separately, where $\beta/\sqrt{k}\rightarrow -\sqrt{r+1}$ for even $d$, and $\beta/\sqrt{k}\rightarrow \sqrt{r+1}$ for odd $d$.
According to whether $\beta/\sqrt{k}\rightarrow \pm\sqrt{r+1}$, the measure \eqref{normalized distribution} converges weakly to $\mu_{\infty}$, where
\begin{align*}
\mu_{\infty}\bigg(\!\!\mp\frac{(-r)^j }{\sqrt{ r+1 }} -\gamma \bigg)={}& \frac{ \big({-}(-r)^{-j-1};-r^{-1}\big)_{\infty} }{ (-r;-r)_j (-r)^{\binom{j}{2}} } \, \bigl({}\mp \gamma\sqrt{r+1}/(-r)^{j+1};-r^{-1}\bigr)_{\infty}
\\
&{}\times \basichypergeometricseries{2}{0}{(-r)^{-j},-(-r)^{-j}}{-}{-r; \pm\gamma(-r)^j\sqrt{r+1}}
\end{align*}
for $j=0,1,2,\dots$.

\begin{Remark}
Set $\gamma=0$ in the above examples, which is the case when we take scaling limits of the vacuum state $\varphi_0$.
The measure $\mu_{\infty}$ is then an affine transformation of the discrete orthogonality measure of the \emph{Al-Salam--Chihara polynomials} (cf.~\cite[Section~3.8]{KS1998R}) for Sections~\ref{subsec: Grassmann} and~\ref{subsec: half dual polar}, that of the \emph{continuous} $q^{-1}$-\emph{Hermite polynomials} (cf.~\cite[Section~3.26]{KS1998R}) for Sections~\ref{subsec: dual polar} and~\ref{subsec: second dual polar}, and that of the \emph{Al-Salam--Carlitz II polynomials} (cf.~\cite[Section~3.25]{KS1998R}) with base $q^{-1}$ for Section~\ref{subsec: sesquilinear forms}.
For the discrete orthogonality measures of the first two families of orthogonal polynomials, see, e.g., \cite[equation~(3.82)]{AI1984MAMS}, \cite[equation~(3.18)]{AK2006ETNA}, and \cite[equation~(6.31)]{IM1994TAMS}.
\end{Remark}

\subsection*{Acknowledgements}

The authors thank the anonymous referees for valuable comments.
HT thanks Professor Tom Koornwinder for letting him know that the measure $\mu_{\infty}$ with $\gamma=0$ in Section~\ref{subsec: Grassmann} corresponds to the Al-Salam--Chihara polynomials, and for providing relevant references.
Part of this work was done while MK was visiting Tohoku University from February to July 2020, supported by K.N.~Toosi University of Technology, Office of Vice-Chancellor for Global Strategies and International Affairs.
NO and HT were supported by JSPS KAKENHI Grant Number JP19H01789.
HT was also supported by JSPS KAKENHI Grant Numbers JP17K05156 and JP20K03551.
This work was also partially supported by the Research Institute for Mathematical Sciences at Kyoto University.

\pdfbookmark[1]{References}{ref}
\LastPageEnding

\end{document}